\documentclass[11pt]{article}
\usepackage{amsfonts}
\usepackage{amsmath,amsthm,amscd,amssymb,mathrsfs,setspace, textcomp,bm}
\usepackage{enumerate}
\usepackage{latexsym,epsf,epsfig}
\usepackage{color}
\usepackage[hmargin=2.25cm,vmargin=2.75cm]{geometry}

\usepackage{hyperref}

\newcommand{\ds}{\displaystyle}

\newcommand{\cA}{{\mathcal{A}}}
\newcommand{\cH}{\mathcal{H}}

\newcommand{\cD}{\mathscr{D}}

\newcommand{\cE}{{\mathcal{E}}}

\newcommand{\cF}{\mathcal{F}}

\newcommand{\Om}{\Omega}
\newcommand{\om}{\omega}

\theoremstyle{plain}
\newtheorem{theorem}{Theorem}[section]
\newtheorem{lemma}[theorem]{Lemma}
\newtheorem{proposition}[theorem]{Proposition}

\newtheorem{definition}{Definition}

\theoremstyle{remark}
\newtheorem{remark}{Remark}[section]
\numberwithin{equation}{section}
\numberwithin{theorem}{section}
\numberwithin{remark}{section}
\numberwithin{assumption}{section}
\numberwithin{condition}{section}
\definecolor{cut}{rgb}{.7,0,.8}

\newcommand{\bbA}{\mathbb{A}}

\newcommand{\bbR}{\mathbb{R}}

\newcommand{\cK}{\mathcal{K}}

\newcommand{\cM}{\mathcal{M}}

\newcommand{\cS}{\mathcal{S}}

\newcommand{\scrH}{\mathscr{H}}

\newcommand{\al}{\alpha}
\newcommand{\bet}{\beta}
\newcommand{\del}{\delta}

\newcommand{\gam}{\gamma}

\newcommand{\eps}{\varepsilon}
\newcommand{\lam}{\lambda}

\renewcommand{\th}{\theta}

\newcommand{\sig}{\sigma}

\newcommand{\zet}{\zeta}

\newcommand{\txtin}{\quad\text{in}\quad}

\newcommand{\txton}{\quad\text{on}\quad}

\newcommand{\txtfor}{\quad\text{for}\quad}
\newcommand{\txtand}{\quad\text{and}\quad}

\newcommand{\txtwith}{\quad\text{with}\quad}

\newcommand{\txtforall}{\quad\text{for all}\quad}

\newcommand{\Rn}{\mathbb{R}^n}

\newcommand{\cl}{\overline}
\newcommand{\half}{\frac{1}{2}}
\newcommand{\fourth}{\frac{1}{4}}

\newcommand{\dom}{\mathscr{D}}

\newcommand{\tl}{\tilde}

\newcommand{\p}{\partial}

\newcommand{\dfn}{\mathrel{\mathop:}=}

\newcommand{\intT}{\int_0^T}

\newcommand{\ltom}{{L^2(0,L)}}

\newcommand{\iffq}{\quad\Leftrightarrow\quad}

\newcommand{\const}{\text{const}\,}

\renewcommand{\lg}{\langle}

\newcommand{\rg}{\rangle}

\newcommand{\vecn}[1]{\begin{bmatrix} #1 \end{bmatrix}}

\title{A Cantilevered Extensible Beam in Axial Flow:\\ Semigroup Well-posedness and Post-flutter Regimes}

 \author{\normalsize \begin{tabular}[t]{c@{\extracolsep{.6em}}c@{\extracolsep{.6em}}c}
      Jason S. Howell &  Daniel Toundykov & Justin T. Webster  \\
\it  Carnegie Mellon Univ. ~~~& \it Univ.~of Nebraska--Lincoln    &\it ~~~ Univ.~of Maryland, Baltimore County \\
\it Pittsburgh, PA & \it Lincoln, NE &\it Baltimore, MD\\
\texttt{howell4@cmu.edu} & \texttt{dtoundykov@unl.edu} &  
\texttt{websterj@umbc.edu}
\end{tabular}}

\begin{document}
\maketitle

\begin{abstract} {\noindent
We consider a cantilevered (clamped-free) beam in an axial potential flow. Certain flow velocities may bring about a bounded-response instability in the structure, termed {\em flutter}. As a preliminary analysis, we employ the theory of {\em large deflections} and utilize a piston-theoretic approximation of the flow for appropriate parameters,  yielding a nonlinear (Berger/Woinowsky-Krieger) beam equation with a non-dissipative RHS. As we obtain this structural model via a simplification, we arrive at a nonstandard nonlinear boundary condition that necessitates careful well-posedness analysis. We account for rotational inertia effects in the beam and discuss technical issues that necessitate this feature. 

We demonstrate nonlinear semigroup well-posedness of the model with the rotational inertia terms. For the case with no rotational inertia, we utilize a Galerkin approach to establish existence of weak, possibly non-unique, solutions. For the former, inertial model, we prove that the associated non-gradient dynamical system 
has a compact global attractor. 
Finally, we study stability regimes and  {\em post-flutter} dynamics (non-stationary end behaviors) using numerical methods for models with, and without, the rotational inertia terms. 
  \\[.3cm]
\noindent {\bf Key terms}: Krieger beam, cantilever, flutter, aeroelasticity, semigroup, global attractor, finite elements
 \\[.3cm]
}
\end{abstract}

\maketitle

\section{Introduction}
Axial flow flutter of a beam or plate is a topic of great recent interest in the engineering literature \cite{b:paidoussis:98,tan-yam-dow:03, tan-pai:07:JSV,tan-pai:08:JSV2,tan-pai:08:JSV,langre1,langre2}, as well as, very recently, the mathematical literature \cite{bal} (and many references therein). In general, the so called {\em flutter phenomenon} may occur when a thin elastic structure is immersed in a fluid flow. For particular flow parameters the onset (a bifurcation) of a dynamic instability resulting from the feedback coupling of the structure's natural elastic modes and the aerodynamic loading may occur. In the context of beams or thin plates, this can happen when the unperturbed flow runs {\em normal} to the principal axis of the beam---{\em normal flow}, or {\em along} the principal axis\footnote{The clamped end is known as the {\em leading edge} and the free end is known as the {\em trailing edge}.} of the beam---{\em axial flow}. The axial flow flutter of a cantilevered beam is an interesting and difficult problem. As with all {\em flutter problems}, the {\em onset} of instability can be studied from the point of view of a linear structural theory---typically as an eigenvalue problem, see \cite{jfs,vedeneev} for recent discussions. Indeed, if one can determine the dynamic load across the surface of the structure, the issue is simply to determine under what conditions this loading will ``destabilize" the natural structural modes. However, if one wishes to study the dynamics in the post-flutter regime, the analysis will require {\em some} physical nonlinear restoring force that will keep solutions {\em bounded in time} \cite{survey1,survey2}. In particular, to study flutter dynamics for a given system, the theory of nonlinear dynamical systems is often utilized, e.g.,  compact global {\em attractors}  \cite{springer, survey1, survey2}.

In the model addressed here, there is a nontrivial interaction between our standard choice of a restoring force and the incorporation of the free boundary condition at the trailing edge. Additionally, it is empirically known that, for cantilevered structures in axial flow, the onset of flutter can occur for very low flow velocities. The characteristic displacements of structures in this configuration empirically fall outside the realm of {\em large deflection} theory \cite{lagleug,inext1,inext2}.  Lastly, in the cantilevered configuration, the aerodynamics near the free end of the beam are highly nontrivial \cite{tan-pai:07:JSV,tan-pai:08:JSV2}. Thus, there are significant challenges in the modeling and analysis of flow driven cantilevered beams. 

{\em Our goal in this paper is to provide a mathematical analysis of a fluttering cantilevered beam model.} We focus our attention on a simple model, whose analysis will already be quite challenging\footnote{We do not claim that this model is the most accurate physically speaking (see Remark \ref{inext} below).}:  {\em an extensible, nonlinear beam} \cite{dickey}; we address the technical challenges associated with obtaining solutions and providing some qualitative discussion. Specifically, the presence of the nonlinearity in this configuration necessitates control of boundary traces at the trailing edge of the form $w_t(L)$---a unique feature. To address such a term rigorously we consider rotational inertia effects in the filaments of the beam, as this will provide additional smoothness of  the velocity $w_t$. Note that one can also address this issue by imposing frictional boundary damping at at $x=L$, see \cite{beam1, beam2, beam3, beam4}. Though theoretical results include rotational inertia effects, we do provide some discussion of the existence of solutions without inertia, as well as interesting numerical simulations in that scenario.

The analysis in this   treatise, apart from its novel mathematical features, serves as a valuable baseline for future studies of more complex dynamics, for instance, inextensible dynamics, as described in Remark \ref{inext} below. The theoretical and numerical work herein can be utilized in comparative studies with more complex dynamics. Such comparisons might support using the {\em much simpler} extensible model as an admissible simplification of the inextensible model in flutter studies.  
It ought also be noted that although the discussion focuses on beams, as they sufficiently demonstrate the key issues at play, much of the analysis here directly applies to two dimensional Berger plates, see Section \ref{mains}.

\begin{remark}[Inextensible dynamics]\label{inext}
Perhaps the most important distinction between cantilevered structures in axial flow, versus fully restricted structures (completely clamped or hinged) in such a flow, is the effect of extensibility \cite{lagnese}. In the case of {\em extensible} beams, transverse deflection necessarily leads to (local) stretching, which is a principal contributor to the elastic restoring force; in the case of a clamped-free condition, the engineering literature indicates that the beam should be taken to be {\em inextensible} \cite{paidoussis, inext1}. The property of inextensibility is best characterized as local arc length preservation throughout deflection. Letting $u$ and $w$ correspond respectively to the in-plane and out-of-plane (Lagrangian) deflections, the condition manifests itself in the requirement:
$\ds 1=\sqrt{\left(1+u_x\right)^2+w_x^2}.$
This relation is then typically linearized to produce what \cite{paidoussis, inext1} refer to as the {\em inextensibility} condition/constraint:
$\ds
u_x+\frac{1}{2}[w_x]^2\equiv 0.$
One can consider inextensible cantilevered beams (and plates), as has been done relatively recently in the engineering literature \cite{inext1,inext2}, {\em albeit there is no rigorous theory available}. Indeed, the lack of mathematical analysis is due to the fact that the constraint leads to complex nonlinear and nonlocal equations of motion. 
\end{remark}

\subsection{Flow effects}
Though there are various ways to consider flow-beam coupling \cite{tan-pai:07:JSV,tan-pai:08:JSV2}, perhaps the most pervasive is to eliminate the fluid dynamics entirely. Such an elimination has the benefit of reducing to a single (non-conservative) dynamics, and eliminates the issues associated with a free or moving boundary for the fluid. The elimination of the fluid can be rigorous, or at least semi-rigorous\footnote{This involves solving the fluid equation for a boundary input, given by the structural displacement, taking the trace of the resulting solution, and feeding the result into the structural dynamics as a force or pressure.} \cite{delay, springer,dowell} or by employing the assumptions of {\bf piston theory}\footnote{Also known as the {\em law of plane sections} \cite{bolotin}.} \cite{dowell, jfs, vedeneev, survey1} for a beam in a potential flow. In the case of large flow velocities, the fluid pressure $p(x,t)$ on the top surface of the structure can be approximated (point-wise) by the fluid pressure on the head of a piston moving through a column of fluid. The dynamic pressure can be written in terms of the {\em down-wash} of the fluid, $d$. This results in the following nonlinear expression:
\begin{equation}\label{actualpist}
  p_*(x,t) = p_0 - \mu \left[1+\dfrac{d}{U} \right]^{\gamma},
\end{equation}
where $d$ is the downwash of the flow, $U >1$ is the (normalized) unperturbed flow velocity, $\gamma>1$ and $\mu>0$ are physical parameters \cite{pist1,pist2}.
 Letting $w$ be the out-of-plate displacement of the beam, the downwash is given by $d= [\partial_t+U\partial_x]w$ on the surface of the beam. Simplifying \eqref{actualpist} via a linearization, and adjusting $p_0$, we obtain the following {\em piston-theoretic} pressure term:
\begin{equation}\label{linpist}
p_1(x,t)=p_0-\beta(w_t+Uw_x),
\end{equation}
which is valid for $U$ sufficiently large (for instance if $U>\sqrt 2$). Here $\beta>0$ is a physical parameter that typically depends on $U$, though, for this mathematical study we have formally decoupled $\beta$ and $U$.

\subsection{Structural models}
For the structure, we begin with a ``large deflection" model for a beam, which, in addition to standard elasticity assumptions, invokes a quadratic strain-displacement law \cite{lagleug}. This model takes into account both in-plane and out-of-plane dynamics, and is {\em extensible}, taking into account the local effect of stretching on bending. Let $(w,w_t)$ correspond to out-of-plane dynamics, and $(u,u_t)$ the in-plane dynamics. The terms $g_i(x)$ are ``edge" forces on the beam. The coefficient $\alpha \ge 0$ below represents rotational inertia in the filaments of the beam\footnote{Often, when considering plates (2-D), $\alpha$ is taken to be zero \cite{lagnese}.}. The terms $D_1, ~D_2>0$ are elastic coefficients.

  \begin{equation}\label{LLsystem}
\begin{cases} u_{tt} -D_1\left[u_x+\frac{1}{2}(w_x)^2\right]_x=0 \\
(1-\alpha\partial_x^2)w_{tt}+D_2\partial_x^4 w -D_1\left[ w_x(u_x+\frac{1}{2}w_x^2)\right]_x = p(x,t) \\
u(t=0)=u_0;~~u_t(t=0)=u_1 \\
w(t=0)=w_0;~~w_t(t=0)=w_1 \\
u(0)=0;~~D_1\left[u_x(L)+\frac{1}{2}w_x^2(L)\right]=g_1(x) \\
w(0)=w_x(0)=0;~~D_2w_{xx}(L)=0;\\
-\alpha\partial_xw_{tt}(L)+D_2\partial_x^3w(L)-D_1w_x(L)\left[u_x(L)+\frac{1}{2}w_x(L)^2\right]=g_2(x). \end{cases}
\end{equation}
\begin{remark}
 In the unscaled version of the equations $\ds D_1=\dfrac{E}{\rho}$, $\ds \alpha=\dfrac{I}{A}$, and $\ds D_2=\dfrac{EI}{\rho A}$, where $\rho$ is the mass density (per unit volume) of the beam, $I$ is the beam's moment of inertia w.r.t. the $y$-axis, $E$ is the Young's modulus, and $A$ is the cross-sectional area of the beam at rest.  
\end{remark}
The principal model under consideration here can be viewed as a simplification of the system above when we take in-plane accelerations to be negligible. Indeed, we can assume that $u_{tt} \approx 0$, and also take the edge forces to vanish $g_i \equiv 0$. Then the first equation above becomes
$$D_1\left[u_x+\frac{1}{2}(w_x)^2\right]_x=0,$$ which can be rewritten as 
$u_x+\frac{1}{2}w_x^2 = c(t).$ Integrating in $x$ from $[0,L]$, we have
$$u(L)-u(0)=c(t)L-\frac{1}{2}\int_0^L w^2_x(\xi)d\xi.$$
We can impose the assumption that the in-plane displacements at the free end of the beam must remain approximately fixed: 
$u(0,t)=0,$ and $u(L,t)=C$, where $C>0$ represents initial in-plane stretching, and $C<0$ compression. As a result, we see that $\ds c =\frac{C}{L}+ \frac{1}{L}\int_0^L w_x^2(\xi)d\xi.$ Plugging this back into \eqref{LLsystem}, we obtain the model of interest here \eqref{Bergerplate}, which we now formally introduce. 

\begin{remark}
As mentioned in \cite{beam1,beam2}, the beam equation (which did not consider a free boundary condition, nor rotational inertia) originally studied was given by \cite{wonkrieg}:
\begin{equation}\label{krieger}
w_{tt} +\dfrac{EI}{\rho}\partial_x^4 w-\left(\dfrac{H}{\rho}+\dfrac{EA}{2\rho L}\int_0^L|w_x|^2 dx\right)w_{xx}=0,
\end{equation} where $L, E, I, A,$ and $\rho$ are as before, and  $H$ is the tension per unit length in the rest position.
\end{remark}

\subsection{Model under consideration}
	
The model we consider is a clamped-free extensible beam model with rotational inertia effects, taken with a (linear) piston-theoretic RHS. For convenience we have  adjusted the parameter names.

 \begin{equation}\label{Bergerplate}
\begin{cases} (1-\alpha\partial_x^2)w_{tt}+D\partial_x^4 w +k_0(1-\alpha\partial_x^2) w_t+(b_1-b_2\|w_x\|^2)w_{xx} = p_0(x)-\beta(w_t+Uw_x) \\
w(t=0)=w_0;~~w_t(t=0)=w_1 \\
w(0)=w_x(0)=0;~~w_{xx}=0; \\
-\alpha\partial_x[w_{tt}+k_0w_t]+D\partial_x^3 w+(b_1-b_2\|w_x\|^2)w_x=0~\text{ at } x=L. \end{cases}
\end{equation}
Above, $b_1 \in \mathbb R$ represents in-plane stretching ($b_1<0$) or compression ($b_1>0$) at equilibrium; $b_2$ is the physical parameter that measures the effect of stretching on bending (i.e., the strength of the nonlinear restoring force). For our theoretical analysis below, we will mostly consider $\alpha>0$---this will be a rather indispensable condition. However, we do provide some discussion of the ``non-rotational'' case $\alpha =0$, as well as numerical investigations concerning the effect of $\alpha$. 

The term $k_0$ measures the strength of structural damping in the beam; when $\alpha=0$, this damping is purely {\em frictional} in nature, and for $\alpha>0$ it is of the {\em ``square root"} type. 

\begin{remark}[Square root damping]\label{sqrt}
Note that when rotational inertia is present, the natural structure of the damping $k_0(1-\alpha\partial_x^2)w_t$ is tailored to the inertial term $(1-\alpha\partial_x^2)w_{tt}$. For our theoretical results on long-time behavior, this is necessary, though we investigate the issue further using numerical techniques. Later, in Section \ref{numerics}, we will  decouple the damping term from the inertial operator to study them as independent parameters.
Formally the damping term $\p_x^2 w_t$ is ``half'' the order of the principal stress operator $B=\p_x^4$, accounting for the  boundary conditions. This concept can be generalized to fractional powers $[B]^\th w_t$ for $\th\in[0,1]$, which at the two extremes yield the usual viscous damping for $\th=0$ and  visco-elastic (Kelvin-Voigt) damping at $\th=1$.  The corresponding square root scenario $\th=1/2$ for a  system of elastic type was proposed in \cite{che-rus:82:QAM} as  this feedback turns out to reproduce energy decay rates empirically observed elastodynamics. This  model was thoroughly investigated in \cite{che-tri:88, che-tri:89:PJM} for elasticity, demonstrating, in particular, that the ensuing evolution semigroup is  analytic if and only if $\th \geq \half$.  The square root-like damping $\p_x^2 w_t$ also arises naturally in other beam models. Consider, for instance,  the Mead-Markus
model \cite{mea-mar:69:JSV} for a  sandwich beam \cite{fab-han:01:DCDS, raj-han:05:DCDS}. In the notation of \cite{raj-han:05:DCDS}, one may: (i) rewrite the shear $s$ as the sum of effective angle $\xi$ and $w_x$, (ii) combine the equations, (iii) compute $s_x$,  and (iv) substitute the result into the out-of-plane equation to rewrite $\xi_{xxx}$. This  explicitly recovers the term $\p_x^2 w_t$ in the out-of-plane displacement equation. Such a representation of the system was explored in \cite{han-las:00} to study stability properties and the analyticity of the semigroup.
\end{remark}

\begin{remark}[Rotational inertia]
From the physical point of view, there is  a disparity between the inclusion of the rotational inertia and the aerodynamic damping provided by piston theory: indeed, the piston-theoretic RHS provides viscous damping of the form $\beta w_t$ (with ``correct'' dissipative sign). However, as per the discussion above, if one includes rotational inertia, this damping from the flow is not appropriate to uniformly stabilize trajectories to some fixed, bounded set in the state space.
\end{remark}

In an auxiliary fashion, let us mention the standard Rayleigh beam with clamped-free conditions:
\begin{equation}\label{linearplate}
\begin{cases} (1-\alpha\partial_x^2)w_{tt}+D\partial_x^4w +k_0(1-\alpha\partial_x^2)w_t = p_0(x)-\beta(w_t+Uw_x) \\
w(t=0)=w_0;~~w_t(t=0)=w_1 \\
w(0)=w_x(0)=0;~~w_{xx}(L)=0,\quad-\alpha\partial_x[w_{tt}(L)+k_0w_t]+D\partial_x^3w(L)=0. 
\end{cases}
\end{equation}
The intrinsic parameters $D$, $L$, and $\beta$ will not be central to our analysis below. For more in-depth numerical parameter analyses see \cite{comppaper}. Thus, we consider these as fixed positive constants and will not frequently mention them hereafter. The parameter $\alpha \ge 0$ is the central parameter in our well-posedness analysis, and, the parameters $U, k_0$ and $\alpha$ are central to discussions of stability and long-time behavior of solutions. 

The onset of flutter in the applied setting can be viewed as the self-excitation of the solutions to \eqref{linearplate} (due to a bifurcation of the dynamics in the parameters $U$, $L$, or $\beta$). In studying the qualitative properties of a fluttering structure, one must invoke a nonlinear restoring force (as in \eqref{Bergerplate}) and can utilize the theory of dynamical systems. In particular, the non-stationary end behaviors of trajectories can be viewed in terms of a {\em compact global attractor} for the dynamics, if one exists. 

\begin{remark}
A natural question would ask about the stability of the nonlinear model ($b_2 \neq 0$) in the presence of {\em standard linear} clamped-free boundary conditions (as in \eqref{linearplate}). We assert that this combination is non-physical, as there is no {\em conservation of energy} associated to the dynamics. Indeed, as is shown in \cite{comppaper}, one can achieve arbitrary growth (in time) of displacements (or energies) for finite difference approximate solutions to:
\begin{equation}
\begin{cases} w_{tt}+\partial_x^4w-\|w_x\|^2w_{xx}  = 0 \\
w(t=0)=0;~~w_t(t=0)=cx \\
w(0)=w_x(0)=0;~~w_{xx}(1)=0,~~\partial_x^3w(1)=0, \end{cases}
\end{equation}
for $c$ sufficiently large. Specifically, if $c=12$, the dynamics exhibit periodic behavior; when $c=13$, the dynamics grow with exponential rate. When the boundary conditions are adjusted appropriately (i.e., $\partial_x^3w(1)=||w_x||^2w_x(1)$), the nonlinear energy $\mathcal E(t)$ (defined below) is nearly {\em perfectly conserved}.
\end{remark}

\subsection{Definitions of spaces and solutions}
Throughout the paper, the notation $(\cdot,\cdot)$ and $\|\cdot\|$ will stand for, respectively, the inner product and the norm in $\ltom$, though, when there is no ambiguity, ``$(p,q)$'' may also refer to an ordered pair. We proceed  with the (now-standard in large deflection beam theory) ``displacement'' state space:
\[
  H^2_* = \{ v \in H^2(0,L) :  v(0) =0,\quad v_x(0) = 0 \}
\]
equipped with an equivalent inner product
\begin{equation}\label{H2*prod}
   (v,w)_{H^2_*} = D (v_{xx}, w_{xx}).
 \end{equation}
 Let $R$ denote the Riesz isomorphism $H^2_*\to [H^2_*]'$ given by:
 \begin{equation}\label{def:R}
 R(v)(w) \dfn  (v,w)_{H^2_*}\,.
 \end{equation}
Note that the same framework may be obtained by introducing the operator
\[
  \cA_0 : \dom(\cA_0)\subset \ltom \to \ltom
\]
\begin{equation}\label{def:cA0}
  \cA_0 f \dfn D\p_x^4 f,\quad \dom(\cA_0) = \{ f \in H^4(0,L) : f(0) =f_x(0)=0, \; f_{xx}(L)=f_{xxx}(L) = 0\}
\end{equation}
\[
  \dom(\cA_0^{1/2}) = H^2_*, \quad \dom(\cA_0^{-1/2}) = [H^2_*]' \txtand \cA_0^{1/2} = R \quad \text{(the Riesz isomorphism \eqref{def:R})}.
\]
Then we could define $(u,v)_{H^2_*}$ as the extension of $(\cA_0 u, v)$ from $\dom(\cA_0)$ to $H^2_*$ which gives \eqref{H2*prod}. 

Next, introduce 
\[
   H^1_* = \{v \in H^1(0,L) : v(0) = 0\},
\]
Topologically we could identify $H^1_*$ with $\dom(\cA_0^{1/4})$, but for the purposes of this discussion it helps to define the  topology on $H^1_*$ as induced by a second-order elliptic operator:
\begin{equation}\label{def:C}
  C f=  -\p_{xx} f,\quad \dom(C) \dfn \{ f\in H^2(0,L) : f(0)=0,\; f_x(L) = 0\}.
\end{equation}
Then $H^1_*\cong \dom(C^{1/2})$. The associated equivalent inner product is given by $(u,v)_{H^1_*} = (Cu,v)$ if $u\in \dom(C)$, and it extends by density to 
\[
(u,v)_{H^1_*} = (u_x,v_x).
\]
We introduce the following additional notation to simplify the exposition below:
\[
  L_{\alpha} = \begin{cases} H^1_* & \text{ for }~\alpha>0  \\ L^2(0,L) & \text{ for }~\alpha =0. \end{cases}
  \]
Then define operator
\begin{equation}\label{def:Cal}
  C_\al = I+\al  C
\end{equation}
We can  define an equivalent inner product on $L_\al$ by means of
\begin{equation}\label{H1-prod}
  (u,v)_{L_\al}= (C_\al^{1/2}u, C_\al^{1/2} v), ~~ \|w\|^2_{L_{\alpha}}:=\alpha \|w\|^2_{H^1_*}+\|w\|^2_{L^2(0,L)}.
\end{equation}
Using the above spaces we equip $\scrH_{\alpha} = H^2_*\times L_\al$ with the product: $y=(y_1,y_2), z=(z_1,z_2) \in \mathscr H_{\alpha}$
\begin{equation}\label{Hal-prod}
  (y,z)_{\scrH_\al} = (y_1,z_1)_{H^2_*} + (y_2,z_2)_{L_\al}.
\end{equation}
Accordingly, the state space for the problem \eqref{Bergerplate} will depend on the presence of the rotational inertia terms (determined by the coefficient $\al$), and will be given by
  \begin{equation}\label{def:state-space}
  \scrH_{\alpha} \dfn H^2_* \times L_{\alpha},
\end{equation}
with the norm
\begin{equation}\label{def:norm}
\|(w_0,w_1)\|^2_{\mathscr H_{\alpha}} =  D\|w_{0,xx}\|^2+\alpha \|w_{1,x}\|^2 + \|w_1\|^2, ~~\alpha \ge 0.
\end{equation}
\begin{definition}[Definite energy]
\begin{equation}\label{def:energy2}
  \cl{E}(t) \dfn \half \|w_t(t)\|^2 + \frac{\al}{2} \|w_{tx}(t)\|^2 + \half D \|w_{xx}(x)\|^2 + \fourth b_2\|w_x(t)\|^4
\end{equation}
\end{definition}
\begin{definition}[Total energy]
\begin{equation}\label{def:energy}
  \cE(t) \dfn \half \|w_t(t)\|^2 +\frac{\al}{2} \|w_{tx}(t)\|^2 + \half D \|w_{xx}\|^2 + \fourth b_2\|w_x\|^4 - \half b_1\|w_x\|^2.
\end{equation}
\end{definition}
The above two energy functionals obey the following  estimates:
\begin{proposition}[Energy comparison]\label{prop:compare}
If $b_1< 0$, then
\begin{align}\label{energy-compare1}
  \half \cE(t) & - \frac{b_1^2}{8b_2}\leq  \cl{E}(t)  \leq \cE(t),~~b_1 <0\\[.2cm]
\label{energy-compare2}
\cE(t) & \leq \cl{E}(t) \leq 2 \cE(t) + b_1^2/b_2,~~b_1>0.
\end{align}
\end{proposition}
\begin{proof}
  The result follows by looking at the minimum values of polynomial functions $\frac{1}{8} b_2 s^4 - \fourth |b_1| s^2$ for $b_1<0$ and $\fourth b_2 s^4 -  b_1 s^2$ for $b_1>0$: respectively $-b_1^2/(8b_2)$ and $-b_1^2/b_2$. 
\end{proof}

We now discuss the pertinent notions of solution. Informally, we have that:

   {\em Weak solutions} satisfy a variational (in time and space) formulation of \eqref{Bergerplate}. One of the key features of such solutions is that the principal, second-order time derivative  $w_{tt}$, is interpreted only in distributional sense. 

   {\em Strong solutions} will refer to weak solutions possessing additional regularity which permits a classical point-wise interpretation of the second-order evolution \eqref{Bergerplate} (albeit, the boundary condition may still have to be interpreted weakly due to subtleties in the case $\al>0$). Such solutions will only be considered for the rotational model, $\al>0$.

   {\em Generalized solutions} are  $C([0,T],\scrH_\al)$ limits of strong solutions. They are also weak solutions, but here the distinction is made that they admit smooth approximations, and thus potentially inherit some properties of strong solutions that hold with respect to the topology of $\scrH_\al$. On the other hand, when talking about weak solutions in general, we do not assert that strong solutions exist to begin with.

   {\em Semigroup solutions} do not technically form a separate class. Rather, this term will refer to the case when the  dynamics can be represented as a semi-flow with an  $\om$-m-dissipative evolution generator. Strong solutions will correspond precisely to those whose initial data resides in the appropriate domain  $\dom(\bbA)\subset \scrH_\al$, invariant under the flow, of that generator. For semigroup initial data only in $\mathscr H_{\alpha}$, the corresponding function is a generalized solution (and hence, also weak).

We now give precise definitions of the solutions discussed above:
\begin{definition}[Weak solution]\label{def:weak}
  We say a function $w \in H^1(0,T; \ltom)$ is a weak solution of \eqref{Bergerplate} on the interval $[0,T]$ if
  \[
    w \in L^\infty(0,T; H^2_*) \cap C_w([0,T]; H^2_*);~~
    w_t \in L^\infty(0,T; L_\al) \cap C_w([0,T]; L_\al).\;\; \footnote{$C_w$ denoting weakly continuous functions.}
  \]
  In addition, for every $\phi \in H^2_*$ we have
  \begin{equation}\label{weak-sol}
    \begin{split}
      & \frac{d}{dt} \left[(w_t, \phi)  + \al (w_{tx}, \phi_x)\right] + (w,\phi)_{H^2_*} + (b_2 \|w_x\|^2-b_1)(w_x,\phi_x)
      + k_0(w_t,\phi) + \al k_0 (w_{tx},\phi_x)\\
   =& (p_0, \phi) -\bet (w_t, \phi) +\bet U (w_x,\phi),
 \end{split}
  \end{equation}
  where $d/dt$ denotes a derivative in the sense of distributions $\cD'(0,T)$. Moreover, for any $\chi \in H^2_*$, $\psi\in \ltom$
  \begin{equation}\label{weak-ic}
    (w,\phi)_{H^2_*}\big|_{t\to 0^+} = (w_0, \chi)_{H^2_*},\quad (w_t,\psi)\big|_{t\to 0^+} = (w_1,\psi).
  \end{equation}
\end{definition}
More regular solutions will be considered in the following sense:
\begin{definition}[Strong solution]
  We say a weak solution on $[0,T]$ is strong if it possesses the following additional regularity: $w\in L^\infty(0,T; W)$, where 
\begin{equation}\label{def:W}
  W \dfn \{ v \in H^2_*\cap H^3(0,L) : v_{xx}(L) = 0\},
\end{equation}
    and $w_t\in L^\infty(0,T; H^2_*)$. In addition $\frac{d^+}{dt^+}w_t$ is right-continuous and $L^\infty(0,T)$ with values in $H^1_*$.  (Note that such solutions still satisfy the dynamic, third-order, boundary  condition only in a weak sense---see \cite{springer} for more discussion of a comparable scenario).
\end{definition}

Finally, the most convenient notion of solutions will be that induced by a semigroup flow:
\begin{definition}[Semigroup well-posedness]\label{def:semigroup}
  We will say \eqref{Bergerplate} is semigroup well-posed if there is a family of  (nonlinear) locally Lipschitz operators $t\mapsto \cS(t)$ on $\scrH_\al$, such that   (i) for any $y_0=(w_0,w_1)\in \scrH_\al$ the function $t\mapsto S(t)y_0$ is in $C([0,T]; \scrH_\al)$, (ii) is a weak solution to \eqref{Bergerplate}, and (iii) is a strong $C([0,T];\scrH_\al)$ limit of strong solutions to \eqref{Bergerplate} on every $[0,T]$, $T >0$. In particular, solutions are unique and depend continuously in $C([0,T];\scrH_\al)$ on the initial data from $\scrH_\al$. Furthermore, there exists some subset, denoted $\dom(\bbA)$ (to be precisely defined below), invariant under the flow such that all solutions originating therein are strong solutions.
  \end{definition}

\section{Main results}\label{mains}
In this section we present the main theoretical results proven in this treatment, and provide some commentary. We do not discuss our numerical analysis of \eqref{Bergerplate} here, instead relegating that to Section \ref{numerics}. Lastly, the relationship between our results and those in the literature is the topic of Section \ref{discussions}.

\begin{remark}[Plate models]
  The Theorems \ref{thm:weak} and \ref{thm:semigroup} below are stated and demonstrated for a Krieger beam.  But exactly the same proofs  extend to the Berger plate model (described in Section \ref{nbm}) as well.
\end{remark}

\begin{theorem}\label{thm:weak}
Let $U,\beta, k_0 \ge 0$, and let $p_0 \in L^2(0,L)$. Then for $\alpha \ge 0$ and any $T>0$ there exists a weak solution to \eqref{Bergerplate}. 
\end{theorem}
\begin{remark}
It is not immediately clear that such solutions (for any value of $\alpha \ge 0$) remain bounded for all time.
In the general case of weak solutions, with $\alpha$ possibly zero (no rotational inertia effects), we make no assertion concerning the uniqueness or continuous dependence upon data.
\end{remark}
In the model with rotational inertia, we can prove something much stronger, namely {\em semigroup well-posedness} of the problem. Depending on the regularity of the initial data, this will yield {\em generalized solutions}  or {\em strong solutions}.
\begin{theorem}[Semigroup flow]\label{thm:semigroup}
Suppose $\al>0$,

Then  system \eqref{Bergerplate} is semigroup-wellposed in the sense of Definition \ref{def:semigroup}, with the domain of the evolution generator given by:
\[ \label{domain}
  \dom(\bbA) = \{ (w,v) : w \in H^2_*\cap H^3(0,L), w_{xx}(L) = 0, v\in H^2_*\}.
\]
In addition, every such solution satisfies the energy identity
\begin{equation}\label{ei}
  \begin{split}
  &  \cE(T) + (k_0 + \bet) \intT \|w_t(s)\|^2 ds+ \al k_0 \intT\|w_{tx}(s)\|^2 ds \\
    =& \cE(0)+\intT\big(p_0 - \bet U w_x(s),\; w_t(s)\big) ds \txtforall T\geq 0.
  \end{split}
\end{equation}
\end{theorem}

Now, we move into the results on long-time behavior. In the case $\alpha>0$, the semigroup well-posedness presented above guarantees the existence of a dynamical system $(\mathscr H_{\alpha},\mathcal S(t))$. We direct the reader to the Appendix for basic terminology and facts about the dynamical systems analysis herein. 
\begin{theorem}\label{attractor}
Let $U,\beta \ge 0$, suppose $\alpha>0$, and take $k_0>0$.  Then the dynamical system $(\mathscr H_{\alpha},\mathcal S(t))$ has a compact global attractor $\mathbf A \subset \mathscr H_{\alpha}$.
\end{theorem}
As described above, the attractor in Theorem \ref{attractor} ``determines", in some sense, the flutter behavior of the model. In particular, as a set, it contains the stationary points of they dynamics, but it is not necessarily identified with the unstable manifold, as the dynamics here are non-gradient. Thus the physical, flutter behaviors can be viewed in the context of convergence of trajectories to this invariant, compact subset of the state space $\mathscr H_{\alpha}$. Later, we will study the attractor numerically by investigating a large class of trajectories for various parameters. Note that such a study can be performed numerically for $\alpha=0$, though we cannot demonstrate the existence of an attractor (nor a proper dynamical system) in that case.

Theorem \ref{attractor} above follows directly---via Theorem \ref{dissmooth}---from two supporting propositions that will be proved individually:
\begin{proposition}\label{ball}
Let $U,\beta \ge 0$, suppose $\alpha>0$, and take $k_0>0$.  Then the dynamical system $(\mathscr H_{\alpha},\mathcal S(t))$ is ultimately dissipative, in the sense that there exists a bounded absorbing set $\mathbf B \subset \mathscr H_{\alpha}$ for the dynamics $\mathcal S$.  
\end{proposition}
\begin{proposition}
Let $U,\beta \ge 0$, suppose $\alpha>0$, and take $k_0>0$. Then the dynamical system $(\mathscr H_{\alpha},\mathcal S(t))$ is asymptotically compact.
\end{proposition}
\begin{remark}
In practice, we will show that $(\mathscr H_{\alpha},\mathcal S(t))$ is {\em asymptotically smooth}, using the criterion in Theorem \ref{psi}. For dissipative dynamical systems, the two properties are equivalent (see the Appendix). 
\end{remark}
Lastly, we note that {\em global-in-time boundedness}  for weak solutions of \eqref{Bergerplate} with $\alpha=0$ follows {\em formally} from the steps in the proof of Proposition \ref{ball} above for $\alpha=0$. We elaborate on this below.

\section{Previous Literature and Relationship with Current Analysis}\label{discussions}
We now provide a thorough literature survey, with three primary focal points: Section \ref{nbm} referencing analyses of various {\em extensible} beam and plate models; Section \ref{bmcb} with a collection of recent papers that look specifically at cantilevered extensible beams, closely related to the model here; and Section \ref{fs} with {\em very recent} studies in the engineering literature which consider the effects of {\em inextensibility} in structural models immersed in axial flow. 

\subsection{Nonlinear beam and plate models}\label{nbm}
We begin by noting that, for the linear beams, when $\alpha=0$, the general form of equation \eqref{Bergerplate}  is known as the Euler-Bernoulli beam, while for $\alpha>0$, one has the so called Rayleigh beam\footnote{See \cite{slenderness} for a nice comparison between four principal linear beam theories across various configurations.}. All nonlinear models considered here are so called {\em large deflection} beam or plate models.  We point out some ambiguity in the terminology associated with these nonlinear models. The vectorial beam equation given as \eqref{LLsystem} above is the beam analogue of the so-called full von K\'arm\'an plate system \cite{lagnese}. In neglecting in-plane accelerations, one can simplify both the 1-D or 2-D systems. In the case of the beam, we have demonstrated that simplification, which yields the beam model here in \eqref{Bergerplate}. Generally speaking, this model has been historically referred to in the literature as a Krieger or Krieger-Woinowsky beam \cite{wonkrieg, beam4}, for instance. In the case of the 2-D full von K\'arm\'an system, when one takes in-plane accelerations to be negligible, the so-called scalar von K\'arm\'an equation \cite{lagnese, springer} are obtained. If one further assumes the {\em second strain invariant is small}, the scalar von K\'arm\'an equation simplifies to what is known as a Berger plate equation \cite{berger}. The nonlinearities in the case of both the Berger plate and the Krieger beam have the same nonlocal structure:
\[
  f(w)=\left(b_1-b_2\int_\Om |\nabla w|^2dX\right)\Delta w,
\]
where, in case of a beam, $w=w(x)$ and $\Om=[0,L]$. In this treatment, we will refer to dynamics (beam or plate) with the nonlinear structure above with the name {\em Krieger}.

Now, we note the classic references \cite{ball, ball2,biacrip,dickey, edenmil} that discuss the well-posedness, stabilization, and attractors for extensible (nonlinear) beam models and Berger plates, with  various boundary conditions---typically {\em clamped} or {\em hinged}. The already referenced \cite{lagleug} discusses the vectorial model \eqref{LLsystem}---which includes rotational inertia---and develops a theory of well-posedness through a clever variable change; this allows for a boundary stabilization analysis. See \cite{pue-tuc:96:AMO, tat-tus:97:NoDEA, koc-las:02, lasiecka:99:CPDE,lasiecka:98:SICON}
for well-posedness and stability analyses of the challenging full von K\'arm\'an plate system. The monograph \cite{lagnese} provides a nice account of the elasticity modeling in the theory of large deflections, and discusses boundary stabilization of various models, though the focus is on plates. 
	
	In  \cite{menz} the authors discuss  the convergence of dynamics of the full 1-D von K\'arm\'an \eqref{LLsystem} system (with rotational inertia) to solutions of the Krieger beam (with rotational inertia). The analysis is highly dependent upon boundary conditions taken in \eqref{LLsystem}, and only clamped or hinged boundary conditions are considered. The work \cite{gw} (and references therein) discusses the general impact of taking {\em free} boundary conditions on a portion of the boundary for a Berger plate, investigating the issues associated with the nonlocal nonlinearity and a third-order boundary condition. Last, but not least, we note the paper \cite{cotizelati} which derives a general nonlinear Krieger-like beam from variational principles in the stationary case, and discusses the structure of the set of solutions.

      \subsection{Berger model for cantilevered beams}\label{bmcb}
      Though there is extensive literature for nonlinear beams (and plates), to the best of the authors' knowledge, there is very little mathematical discussion of {\em nonlinear} cantilevered beams. And, as we have pointed out above, to study the dynamics in the {\em post-flutter} regime, one requires some nonlinear restoring force. We do acknowledge that there is a nontrivial literature for nonlinear plates (von K\'arm\'an---scalar and full) with free boundary conditions---see \cite{daniellorena, newirena}, or \cite{springer} (and numerous references therein). 
      
      The principal theme for nonlinear elastic structures with a portion of the boundary free seems to be the need to adjust/augment the boundary conditions to accommodate other aspects of the theory---for instance in the study of long time behavior \cite{survey1,survey2}. In the analysis here, the boundary conditions are augmented (relative to the standard linear boundary conditions) owing to the analysis of the constant of integration, following the reduction from \eqref{LLsystem}.   
      A cantilevered model of this type, though with more general nonlinearity, was first considered in \cite{beam1} (well-posedness and uniform stability), and later in \cite{beam2,beam4}. In \cite{beam2} the author considers the static version of the problem in \cite{beam1}, with the associated nonlinear free boundary condition; that paper includes a very brief numerical study with specific polynomial nonlinear structure.  The subsequent work \cite{beam4} is our primary motivating reference here, and addresses the dynamic equations with a general Berger-like nonlinearity from the point of view of {\em attractors}, following from \cite{beam2,beam3}. The latter analysis utilizes Lyapunov methods to demonstrate the dissipativity of the dynamical system, as well as a version of Theorem \ref{psi} to demonstrate asymptotic compactness.
      
      In each of the above references, however,  {\em rotational inertia effects are neglected, and boundary damping (as well as boundary sources) is included.} Additionally, the focus of these papers is not flutter, and thus the piston-theoretic terms in \eqref{Bergerplate} are not present. Due to the regularizing presence of boundary damping, the notion of solution considered in these papers corresponds to generalized and strong solutions only. In this treatment we also discuss weak solutions for a model without rotational terms (and no boundary damping). 
   We emphasize that the inclusion of boundary damping in the existing papers is precisely to accommodate troublesome trace terms arising in the analysis of the difference of two trajectories---a key component in uniqueness, continuous dependence, and asymptotic compactness of associated dynamical systems. See further comments in Remark \ref{alphazerosmooth} below.

		\subsection{Flutter studies}\label{fs}
		We begin by noting that a bulk of the engineering literature addressing flutter, especially in the axial configuration, omit the effects of the rotational inertia. Indeed, a scaling argument is typically invoked, and for ``thin" structures the term is often neglected \cite{slenderness, lagnese}. Mathematically, however, the presence of rotational inertia is highly nontrivial due to its regularizing effects on the velocity $w_t \in H^1$ for $\alpha>0$. In this paper we note that its presence is central for us to obtain a bulk of our theoretical results in the situation when there is no boundary damping imposed on the free end. We point out that there are other (typically nonlinear) problems where the regularizing effect of rotational inertia for the velocity $w_t$ is paramount, see, for example, \cite{lagleug, newirena} and references therein. 
		
As discussed in Remark \ref{inext}, the key distinction for cantilevered beam models arises in the discussion of the dominant nonlinear effects: extensible---nonlinear effect of stretching on bending, or inexensible---nonlinear inertial and stiffness effects. Future work will address well-posedness of recently derived inextensible beam models \cite{inext1,inext2}, and (both theoretical and numerical) long-time behavior analyses in the presence of aerodynamic loading will follow. See also the papers \cite{paidoussis,paidoussis2,paidoussis3} for further discussions of the need for, and effects, of implementing the inextensibility constraint in the context of fluttering beams, as well as tubes conveying fluid. 

For more general studies of fluttering beams and plates, we point to the classical references \cite{bolotin,dowell}. Specifically, in the discussions of piston theory, we  should mention the modern engineering references \cite{jfs,vedeneev}, classic engineering references \cite{pist1,pist2, bolotin}, and the mathematical surveys \cite{survey1,survey2}. Other mathematical references addressing piston-theoretic models include \cite{hlw, Memoires, daniellorena,springer}. The article \cite{hlw} performs numerical simulations (which guide our work in Section \ref{numerics} below); theoretically, \cite{hlw} addresses various classes of {\em attractors} arising in the clamped beam/plate configuration.  We follow much of the analysis there, in particular tracking stablity/instability of dynamics with respect to piston-theoretic terms, as well as types and size of damping effects.

	\subsection{New goals and challenges}\label{mainpoints}

	In the context of the aforementioned body of work in Section \ref{discussions}, and a no less vast field of open questions in this area, the present article accomplishes several important goals and addresses multiple technical challenges which we now articulate.

	   The reduction, here, of the system \eqref{LLsystem} to a single nonlocal beam equation comes at the cost of a highly nonlinear boundary term. This nontrivial term disallows the (standard) treatment of the interior nonlinearity as a perturbation of the linear dynamics \eqref{linearplate}. The ensuing challenge is that  ($\om$)-dissipativity of the generator, in the presence of nonlinear Neumann condition for a hyperbolic-like problem, generally requires  that \emph{the velocity trace be well-defined for weak solutions.} 
	Traditionally that has been achieved by including boundary damping, e.g., \cite{gra-bel:12:JDE}, as well as \cite{beam1,beam4}. Instead, we opt to include rotational inertia terms to boost the regularity of the velocity, while also demonstrating that the system admits a suitable monotonic structure. The work here---taken in conjunction with \cite{beam1,beam4} (and related references)---further indicates a strong need for some mitigating factor in addressing boundary traces of the form $w_t(L)$ for nonlinear, cantilevered models. 
	    
	    To obtain a semigroup formulation, one {\em must} include the boundary conditions  in the domain the of the principal operator (due to the inextricable link between the nonlinearity and free end at $x=L$). Thus, one has to start with a fully nonlinear  generator, as opposed to writing it as a linear generator with a locally Lipschitz perturbation, as is typical in the theory of large deflections. 
	    
	  The system is not monotonically dissipative, and from the energy law it is not obvious that the energy is globally bounded in time in the first place. The dynamical system associated to solutions is of non-gradient type. While a similar framework has been considered before \cite{daniellorena}, the aforementioned non-linearity in the boundary condition prevents the applicability of the most methods \cite{Memoires, springer} to prove the asymptotic smoothness of the dynamics, and here we must adjust the ``standard" approach. 
	
	 To prompt further investigation of asymptotic regimes of this system, we provide detailed numerical analysis demonstrating the effect of the nonlinearity, critical parameters (e.g., $U$, $\alpha$, and $k_0$). This is particularly pertinent in the case $\alpha=0$, when we have only existence of weak solutions, and no energy identity for the piston-theoretic dynamics. We investigate the post-flutter regime thoroughly for $\alpha \ge 0$, which provides direction and intuition for future studies. 

\section{Well-posedness and Energy Identities}\label{sec:wellposed}
\subsection{$\alpha \ge 0$: Existence of weak solutions}

This section is devoted to the proof of Theorem \ref{thm:weak}. The theorem admits any $\al\geq 0$, but for $\al>0$ we will prove a much stronger result in the sequel, hence this section focuses on $\al=0$ and $L_\al = \ltom$ with the respective state space $\scrH_\al = H^2_*\times \ltom$. The weak formulation \eqref{weak-sol} does not recognize higher-order boundary conditions, so to build such a solution we will use the basis of eigenfunctions of the unbounded positive self-adjoint operator with compact resolvent, $\cA_0$, introduced earlier in \eqref{def:cA0}. Appropriately scaled eigenfunctions $(e_k)$ of $\cA_0$, with respective positive eigenvalues $0<\lam_1 \leq \lam_2 \leq \cdots \leq \lam_k \leq \cdots \to \infty$,  form an orthonormal basis for $\ltom$ and orthogonal basis for every $\dom(\cA_0^{s})$, $s\in \bbR$.  

Using this basis we define smooth approximations
\[
  w^n(t,x) \dfn \sum_{k=1}^n \eta_k(t) e_k(x)
\]
and consider the $n$-system of ODEs in $t$-variable for the coefficient functions $\{\eta_1, \eta_2,\ldots,\eta_n\}$ obtained by replacing $w$ with $w^n$ in \eqref{weak-sol}, and $\phi$ with $e_k$, $k=1,\ldots,n$. Since the basis is orthonormal in $L^2$, the principal part of the system is diagonal (recall, again, that we can restrict to $\al=0$):
\[
  \eta_k''(t) + P_{3,k}(\eta_1(t),\ldots, \eta_n(t)) + (k_0+\bet) \eta_k'(t) = (p_0, e_k) + \sum_{j=1}^n \eta_j (e_{jx},e_k),\quad k=1,\ldots,
\]
where $P_{3,1},\ldots, P_{3,n}$ are cubic polynomials in $n$ variables.  The initial conditions are determined by
\[
  \eta_k(0) =  \frac{1}{\lam_k}(w_0,e_k)_{H^2_*}, \quad \eta'_k(0) = (w_1,e_k).
\]
The nonlinearity is polynomial so this system of ODEs has a local solution on some interval $[0,T_0]$:
\[
  \eta \in C^2([0,T_0]; \Rn).
\]
Let $\cE^n$ denote the full energy functional corresponding to $w^n$, and $\cl{E}^n$ be the corresponding definite energy. Substituting $w^n_t$ for the test function in \eqref{weak-sol} we obtain for $ t \in [0,T_0]$
\[
   \begin{split}
     \cE^n(t)  +  (k_0 +\bet)\int_0^t \|w^n_t(s)\|ds  =\cE^n(0) + \int_0^t (p_0, w^n_t(s)) ds - \int_0^t \bet U(w^n_x(s), w^n_t(s)).
 \end{split}
 \]
 We can invoke Gr\"onwall's inequality, along with the comparison of $\cE^n$ and $\cl{E}^n$ shown in Proposition \ref{prop:compare}, to deduce that for some number $m \geq 0$ we have:
 \[
   \cl{E}^n(t) \leq  C(\cl{E}^n(0)) e^{mt} \txtforall t\in[0,T_0].
 \]
 This estimate on the energy implies that (i) any solution $\eta$, hence $w^n$, has right-maximal interval of existence $[0,\infty)$ in time; (ii) for any $T>0$, the sequence $(w^n)$  is bounded in $C([0,T]; H^2_*)$; and (iii) the sequence  $(w^n_t)$ is bounded in $C([0,T]; \ltom)$.  In particular,  $w^n$ is bounded in $H^1(0,T; \ltom)$ and, up to a subsequence, has a weak limit $w$, which in addition has the following regularity: 
 \[
   w \in H^1(0,T;\ltom) \cap L^\infty(0,T; H^2_*) \txtwith w_t \in L^\infty(0,T; \ltom).
 \]
 We have, 
 \[
   w^n \stackrel{\text{weakly}*}{\to} w  \in L^\infty(0,T; H^2_*),\quad w^n_t \stackrel{\text{weakly}*}{\to} w_t \in L^\infty(0,T; \ltom),
 \]
 and, by Aubin-Simon compactness result \cite{simon:87:AMPA}, for any $\eps>0$,
 \[
   w^n \to w \quad\text{strongly}\txtin L^\infty(0,T; H^{2-\eps}(0,L)).
 \]
Furthermore, because $H^2_*$ and $\ltom$ are separable, up to a subsequence we can claim that 
\begin{equation}\label{weak-pointwise}
  w^n(t) \stackrel{\text{weakly}}{\to} w(t) \txtand w^n_t(t) \stackrel{\text{weakly}}{\to} w_t(t) 
\end{equation}
for almost every $t\in [0,T]$.

 Taking $n\to \infty$ we see that for any $\phi\in H^2_*$ and any $\zet \in C_c^\infty((0,T))$ we have
\[    
 \begin{split}
      & -\intT (w_t, \phi) \zet'(t) dt + \intT (w,\phi)_{H^2_*}\zet(t) dt + \intT (b_2 \|w_x\|^2-b_1)(w_x,\phi_x) \zet(t)
   +\intT  (k_0+\bet)(w_t,\phi) \zet(t)dt\\
  & = \intT (p_0, \phi)\zet(t)dt +\bet U \intT (w_x,\phi)\zet(t) dt
 \end{split}
\]
which shows that $w$ verifies the variational identity \eqref{weak-sol}. From the same identity it follows that 
$ w_{tt} \in \cD'((0,T); [H^2_*]')$ and, moreover, this distribution is defined point-wise a.e.  $(0,T)$ as the element of $[H^2_*]'$ given by
\[
  w_{tt} (t) =- (w,\cdot)_{H^2_*} - (b_2 \|w_x\|^2-b_1)(w_x,\p_x \cdot )
   - (k_0+\bet)(w_t,\cdot)+ (p_0, \cdot) +\bet U(w_x,\cdot).
 \]
 Since $\|w_{tt}(t)\| \leq C \left( \|w\|_{L^\infty(0,T;H^2_*)} \right)$, then we conclude
 \[
   w_{tt} \in L^\infty(0,T; [H^2_*]').
 \]
 The regularity of $w$, $w_t$, and $w_{tt}$ now imply that  \cite[Lemma 8.1--8.2, pp. 275--276]{b:lio-mag:v1}
 \[
   w \in C_w([0,T]; H^2_*) \txtand w_t \in C_w([0,T];\ltom).
 \]
 By construction, the initial data  $w^n(0)$  and $w^n_t(0)$ converge strongly to $w_0$ and $w_1$ in $H^2_*$ and $\ltom$ respectively. The weak point-wise a.e. convergence  \eqref{weak-pointwise} and the weak continuity of $w$ and $w_t$ imply that, for every $\chi \in H^2_*$, $\psi \in \ltom$, and $t\in [0,T]$, we have
 \[
   \lim_{n\to \infty}(w^n(t), \chi)_{H^2_*}= (w(t), \chi)_{H^2_*},\quad  \lim_{n\to \infty}(w_t^n(t), \psi) = (w_t(t), \psi).
 \]
\qed 

\begin{remark} For such solutions, no comments on uniqueness are made, but the situation will improve if rotational effects are included, as discussed in the next section.
\end{remark}

\subsection{$\alpha>0$: Full Hadamard well-posedness}\label{sec:semigroup}

The discussion here is devoted to the proof of Theorem \ref{thm:semigroup} and divided into several subsections.

\subsubsection{Nonlinear elliptic preliminaries}

Consider the following nonlinear functional: 
\begin{equation}\label{def:J}
  J(v) \dfn \half D \|v_{xx}\|^2 + \fourth b_2 \|v_x\|^4  - \half b_1 \|v_x\|^2 \txtfor v\in H^2_*.
\end{equation}
The Fr\'echet derivative of $J$ at $v\in H^2_*$, as an element of the dual space $[H^2_*]'$, is given by
\[
  DJ(v)(\phi) = D(v_{xx},\phi_{xx})  + b_2 \|v_x\|^2(v_x,\phi_x) - b_1 (v_x,\phi_x) \txtforall \phi \in H^2_*\,.
\]
Next, introduce the nonlinear  operator
\begin{equation}\label{def:A}
  A v \dfn  D \p^4_x v - b_2\|v_x\|^2 v_{xx} + b_1 v_{xx}
\end{equation}
with domain
\[
\cD(A) \dfn \{ v \in H_*^2 :  Av\in \ltom,\quad v_{xx}(L) = 0,\quad D v_{xxx}(L)- \|v_x\|^2 v_{x}(L) + b_1v_x(L)  = 0 \}.
\]
From the definition we immediately have that $\dom(A) \subset H^4\cap H^2_*$, 
whence the traces are well-defined. Even without appealing to full $H^4$ regularity, $\p_{xx}v$ is defined for such $H^2_*$ functions by duality: 
\begin{equation}\label{vxx}
  \lg v_{xx}, \psi \rg = (v_{xx},\Psi_{xx}) -(\p_x^4 v, \Psi)
\end{equation}
for any $\psi$ (in higher dimensions $\psi\in H^{1/2}(\p \Om)$) and any $\Psi\in H^2_*$  whose normal trace is $\psi$, for example 
\[
\p_x^4 \Psi \in L^2(0,L),\quad \Psi_x(L) = \psi,\qquad \Psi(L) =0.
\]
The trace $v_{xxx}$ for $v\in H^2_*$ is likewise is given by duality 
\begin{equation}\label{vxxx}
  \lg v_{xxx}, \psi \rg = (\p_x^4 v, \Psi)  - (\p_{xx} v, \p_{xx} \Psi)
\end{equation}
for any $\psi$ (in higher dimensions $\psi\in H^{1/2}(\p\Om)$) and any $\Psi\in H^2_*$ such that 
\begin{equation}\label{dom:A}
\p_x^4 \Psi\in L^2(0,L),\quad \Psi(L) = \psi,\qquad \Psi_{xx}(L) =0.
\end{equation}
Because the Dirichlet traces above order $3/2$ are not bounded with respect to the $H^2_*$ topology, the domain $\dom(A)$ is dense in $H^2_*$. Moreover, for any $v\in\dom(A)$ we have $(Av, \phi) = DJ(v)(\phi)$. 
Consequently if we associate $A$ with its closure in $H^2_*\times [H^2_*]'$ we can say that
\[
   A = DJ\txton H^2_*
\]

An additional observation will be needed due to the later discussion on rotational inertia terms. Recall the definition of $W$:
\[
  W \dfn \{ v \in H^2_*\cap H^3(0,L) : v_{xx}(L) =0\}
\]
Then for  any  $v\in W$ and $\phi \in H^2_*$ we can integrate the higher-order term by parts and obtain 
\[
  (Av,\phi) = - D(v_{xxx}, \phi_x)  + b_2 \|v_x\|^2(v,x \phi_x) - b_1(v_x,\phi_x) 
\]
This action extends continuously to any $\phi \in H^1_*$. Thus we have~ $A: W \to [H^1_*]'.$

\subsubsection{Weak nonlinear problem}

We will take advantage of the following nonlinear existence result: 

\begin{lemma}\label{lem:Av=F} Let $J$ be given by \eqref{def:J}. For any $\cF \in [H^2_*]'$, $\lam\geq 0$, and $\mu\in \bbR$ there exists $v\in H^2_*$ such that
  \begin{equation}\label{weak-eqn}
    \lam (v,\phi)_{H^2_*} + \mu (v,\phi)_{H^1_*}+ DJ(v)(\phi) = \cF(\phi)\txtforall \phi\in H^2_*,
\end{equation}
Or, equivalently, for $R$ as in \eqref{def:R}, $A$ as in \eqref{def:A}, and $C$ as in \eqref{def:C}, there exists a weak solution to
\begin{equation}
(\lam R+ \mu C +A)(v) = \cF  \txtin [H^2_*]'.
\end{equation}
In addition, if $\cF \in [H^1_*]'$, then $v \in W$. And if $\cF$ is given by integration against an $\ltom$ function $f$, then $v\in \dom(A)$.
\end{lemma}
\begin{proof}
We only consider the case $\lam=0$ since the case $\lam>0$ easily follows along the same lines.  Likewise, note that  $\mu  C$ can be merged with $A_M$ by redefining the constant $b_1$. Thus, it suffices to take~
 $\lam = 0 = \mu.$

If $b_1\leq 0$, the functional $J$ and also $v\mapsto J(v) -  \eps \frac{D}{2}\|v_{xx}\|^2$ for $\eps <1$ are both strictly convex.  So the operator $A_\eps \dfn A-\eps R$, $\dom(A_\eps)=\dom(A)$ is  the subdifferential of a convex functional and, thus,  is maximal monotone $H^2_*\to [H_*^2]'$. So 
$A = A_\eps +\eps R$ is a bijection $\dom(A)$ onto $[H^2_*]'$.

In the complementary case $b_1>0$, the convexity, and thus weak closure properties of the epigraph of $A$ are lost, but we can use a Galerkin  argument. Consider the operator $\cA_0$ introduced earlier in \eqref{def:cA0}. This time  we are going to normalize the basis of eigenfunctions $(e_k)$ so that it forms an ONB for $H_*^2$ with respect to the equivalent inner product \eqref{H2*prod}, and orthogonal basis for $L^2(0,L)$.

Let  $V_n$ be the subspace spanned by $\{e_1,e_2,\ldots, e_n\}$. On this space the finite-dimensional approximation of \eqref{weak-eqn} for $v = \sum_{i=1}^n \al_i e_i$ yields the nonlinear system
\begin{equation}\label{eqn:alpha}
  \al +(b_2 |\al|_P^2-b_1) P \al  = F\txtin \Rn\,.
\end{equation}
where $F_{i}\dfn \lg e_i, f_i\rg_{H^2_*\times [H^2_*]'}$, $P_{ij} = (\p_x e_{i}, \p_x e_{j})$ is positive-definite
and $\lg \gam, \del\rg_P = \gam^t P \del$ is an equivalent inner product on $\Rn$.

Note that the functional
\[
  \bet \mapsto \half |\bet|^2 +  \fourth b_2 |\bet|_P^4 -  \half b_1 |\bet|_P^2 - F\cdot \bet\txtfor \bet \in \Rn
\]	
is Fr\'echet differentiable and bounded below on $\Rn$. Hence any vector $\al$ corresponding to its local minimum (possibly non-unique) would yield a solution to \eqref{eqn:alpha}.

Thus, the $n$-dimensional approximation of \eqref{weak-eqn} has a solution $v^{(n)}$ on $V_n$. And satisfies
\[
  D(v^{(n)}_{xx},\phi_{xx})  + b_2 \|v^{(n)}_x\|^2(v^{(n)}_x,\phi_x) - b_1 (v^{(n)}_x,\phi_x) = \cF(\phi) \txtforall \phi \in V_n\subset H^2_*\,.
\]
Choosing $\phi = v^{(n)}$ gives an a  priori bound on $\|v^{(n)}\|$ in $H^2_*$. Hence, up to passing to a subsequence, $v^{(n)}$ converge to some $v$ weakly in $H^2_*$ and strongly in $H^1_*$. Passing to the limit $n\to \infty$ we recover a solution to \eqref{weak-eqn}.
The claimed regularity properties of the solution depending on $\cF$ are inferred by standard arguments relying on the duality representations of the traces \eqref{vxx} and \eqref{vxxx}.
\end{proof}
\subsubsection{Truncations of $A$}
Looking ahead, in order to construct local---at first---solutions, we will need ``truncated'' versions of $A$
in the spirit of \cite{lagleug,chu-ell-las:02:CPDE, boc-las:10:JDE}. Consider the following linear version of the operator $A$ defined in \eqref{def:A}:
\begin{equation}\label{def:L}
  L_M v = D\p_x^4 v - (b_2 M^2 v_{xx}- b_1) v_{xx}\txtwith M^2 > |b_1|
\end{equation}
\[
\cD(L_M) = \{ v \in H_*^2 :  L_M v\in \ltom,\quad   v_{xx}(L) = 0,\quad Dv_{xxx}- (b_2 M^2-b_1) v_{x}= 0 \}.
\]
Classical linear elliptic theory readily yields that $\lam R + L_M$ is surjective $\dom(A_M) \to L^2(0,L)$.
Next, introduce the following ``truncated'' operator
\begin{equation}\label{def:AM}
   A_M v \dfn 
   \begin{cases}
     A v  &  \|v_x\|\leq M\\
     L_M v &   \|v_x\| > M
   \end{cases}
 \end{equation}
\[
\dom(A_M) =\left\{ v \in H^2_* : \|v_x\| \leq M,  v\in \dom(A)\right\}\; \bigcup\; \left\{ v\in H^2_* : \|v_x\|>M, v\in \dom(L_M)\right\}.
\]
\begin{proposition}\label{prop:AM}
  For Riesz isomorphism $R$ in \eqref{def:R} and $A_M$ given by \eqref{def:AM}, the operator $\lam R+ A$ is surjective $\dom(A)\to L^2(0,L)$  and $W\to [H^1_*]'$ for any $\lam \geq 0$.
\end{proposition}
\proof 
This statement follows individually for $A$ from Lemma \ref{lem:Av=F}, and for $L_M$ as a simpler case.  The only obstacle would be if for some $\cF \in [H^2_*]'$ we have
\begin{equation*}
  \lam R v+ Av = \cF \txtand   \lam R w + L_M w = \cF
\end{equation*}
where $\nu \dfn \|v_x\| > M \geq  \|w_x\|$. So $\cF$ would be in the image of $\lam R+A$ and of $\lam R +L_M$, but not in the image of $A_M$. We will show that this scenario cannot happen. The weak formulations of the above identities respectively read:
\[
  (\lam+D)(v_{xx},\phi_{xx}) + (\nu^2 -b_1)(v_x,\phi_x)  =  \cF(\phi) \txtforall \phi\in H^2_*, 
\]
\[
(\lam+D) (w_{xx},\phi_{xx}) + (M^2-b_1)(w_x,\phi_x) =  \cF(\phi) \txtforall \phi\in H^2_*.  
\]
For any particular $\phi$ consider the difference of these identities:
\[
 (\lam +D) (v_{xx}-w_{xx},\phi_{xx}) + ( [\nu^2 - M^2] v_x + [M^2-b_1] (v_x - w_x),\phi_x) =  0\,.  
\]
Substitute $\phi = v-w$ and appeal to the fact that $M^2> |b_1|$ to arrive at:
\[
   ( [\nu^2 - M^2] v_x , v_x-w_x ) \leq 0 \iffq   \|v_x\|^2 \leq (v_x,w_x)
\]
Since $\|v_x\|=\nu>M\geq \|w_x\|$, then we have a contradiction $M<\nu \leq M$.
\qed

\subsubsection{Semigroup generation}

For semigroup analysis purposes we can regard in \eqref{Bergerplate} the interior linear (or, more generally, Lipschitz) terms that are continuous with respect to $\scrH_\al$ topology as a bounded perturbation  $\tl{p}(x,t) \dfn p_0(x) -k_0 w_t - \bet(w_t + U w_x)$ which  does not affect semigroup generation.  First, consider the case when $\tl{p}\equiv 0$. We start by formulating a suitable operator-theoretic analog of the problem. Recall $C$ and $C_\al= I+\al C$ from \eqref{def:C}, \eqref{def:Cal}. For $\al>0$ we can close $C$ and $C_\al$ to  topological isomorphisms from $\dom(C_\al^{1/2}) =\dom(C^{1/2})\cong H^1_*$ to $[H^1_*]'$. Introduce the evolution generator  on $\scrH_\al$
\[
   \bbA \vecn{w\\v}\dfn 
   \vecn{v \\ -C_\al^{-1}Aw - C_\al^{-1} \al k_0 C v } 
\]
Recall the previously defined space $W = H^2_*\cap H^3(0,L)$ subject to the boundary condition $v_{xx} = 0$. Then  we define
\[
  \dom(\bbA) = \{ (w,v) : w \in W, v\in H^2_*\}.
\]
The original system \eqref{Bergerplate} can be represented as an evolution problem on $\scrH_\al$
\begin{equation}\label{y-formulation}
  y' = \bbA y,\quad y(0) = y_0\dfn \vecn{w_0 \\ w_1}\in \scrH_\al.
\end{equation}
First, we consider the ``truncated'' system
\[
  y'_M = \bbA_M (y_M),\quad y_M(0)=y_0\in \cH_\al.
\]
\[
  \bbA_M\vecn{w\\v} \dfn \vecn{v \\ -C_\al^{-1}A_M (w) - C_\al^{-1} \al k_0 C v },\quad \dom(\bbA_M) = \dom(\bbA)
\]
We claim that $\bbA_{M}$ is $\om$-dissipative. Suppose first $y,z\in \dom(\bbA_{M})$ such that $\|\p_x y_1\|\leq M$ and $\|\p_x z_1\| \leq M$. Using the definition of \eqref{Hal-prod} and \eqref{H1-prod} we obtain
\begin{equation}\label{A0M-dissip}
  \begin{split}
  &  (\bbA_{M} y - \bbA_{M} z,  y - z )_{\scrH_\al} \\
  =&~ (b_2 \|y_{1x}\|^2-b_1) (z_{1x}-y_{1x}, y_{2x}-z_{2x}) \\
  & + b_2 (\|z_{1x}\|^2 - \|y_{1x}\|^2) (z_{1x}, y_{2x}-z_{2x})- \al k_0(Cy_2-Cz_2,y_2-z_2) \\
  \leq &~  (b_2 M^2 +|b_1|) \|z_{1x}-y_{1x}\|\| y_{2x}-z_{2x}\| + b_2 M \|z_{1x}+y_{1x}\|\|z_{1x}-y_{1x}\|\|y_{2x}-z_{2x}\|\\
  &+ \al k_0 \gam \|u_2-z_2\|_{L_\al}^2\\
  \leq &~ c (b_2 M^2+|b_1|) \|y-z\|_{\scrH_\al}^2.
\end{split}
\end{equation}
\begin{remark}
  If $\al=0$ then the estimate \eqref{A0M-dissip} cannot be repeated with respect to $\mathscr H_0= H^2_*\times \ltom$ because of the term $\|y_{2x}-z_{2x}\|$ on the RHS. This term corresponds to spatial derivatives of the velocity variables of the solutions, whereas velocities are not in $H^1(0,L)$ unless $\al>0$.
\end{remark}
If we assume that $\|\p_x y_1\| > M$ and $\|\p_x z_1\| \leq M$, then in the first terms after the equal sign in \eqref{A0M-dissip}   (ignoring $\al$) become
\[
  \begin{split}
   (b_2 M^2-b_1) (z_{1x}-y_{1x}, y_{2x}-z_{2x}) + b_2 (\|z_{1x}\|^2 - M^2) (z_{1x}, y_{2x}-z_{2x})
\end{split}
\]
Now using the 	fact that $\|y_{1x}\|^2 - \|z_{1x}\|^2 \geq M^2 - \|z_{1x}\|^2 > 0$ we can estimate them  as
\[
  \cdots \leq  b_2 M^2 \|z_{1x}-y_{1x}\| \| y_{2x}-z_{2x}\| + b_2 (\|y_{1x}\|^2 - \|z_{1x}\|^2) \| z_{1x}\| \| y_{2x}-z_{2x}\|
\]
from which the same same estimate as in \eqref{A0M-dissip} follows. The same calculation holds if $\|\p_x y_1\| \leq M$ and $\|\p_x z_1\|>M$, and, finally, the case when each of the norms exceeds $M$ is immediate as the operators involved in that case are linear.

Thus $\bbA_{M}$ is an $\omega$-dissipative operator on $\scrH$. Let's show that it is also maximal, that is, $-\bbA$ is $m$-accretive. To this end we must solve the system $  \lam y - \bbA_{M}(y) = F \in \scrH_\al$ for any $\lam>0$, say $\lam=1$, which amounts to 
\[
  \vecn{w - v\\ v +C_\al^{-1}A_M (w) + C_\al^{-1} \al k_0 C v  } = \vecn{g \in H^2_*\\ h\in L_\al = H^1_*}.
\]
Substitute $v= w-g$ into the second equation to obtain (also recall $C_\al = I + \al C$)
\[
 w+ \al (1+k_0)C w + A_M (w)   = C_\al h + g + \al(1+k_0) Cg 
\]
Note that because $g\not\in \dom(C)$, then the action $Cg$ is only permitted by using  the closure of $C$ to $H^1_*\times [H^1_*]'$, whence $Cg$ is interpreted as an element of $[H^1_*]'$. Thus we conclude that $C_\al h + g + \al(1+k_0) Cg  \in [H^1_*]'$, and therefore this problem has a solution $w \in W$ as verified by Proposition \ref{prop:AM}. 

So $\bbA_{M}$ is an m-$\om$-dissipative operator on $\scrH_\al$.
Thus, by Kato's theorem (e.g., see \cite[Thm 4.1 and 4.1A, pp. 180, 183]{b:showalter:97} or \cite[Thm. 1.6, p. 216]{b:barbu:93}), the evolution problem
\begin{equation}\label{y-M-formulation}
  (y^M)' = \bbA_{M} y,\quad y^M(0) = y_0
\end{equation}
has a unique solution rendered by a nonlinear semigroup flow $t\mapsto \cS(t) y_0$ on $\scrH_\al$. Such solutions reside in $C([0,T]; \scrH_\al)$ for any $T>0$, and moreover, can be approximated in this topology by strong solutions whose initial data resides in the (dense) domain $\dom(\bbA_M)$.

We still need to link the ``truncated'' generator $\bbA_M$ with solutions of the original system.
A simple optimization argument shows that $b_2>0$ and any $b_1\in \bbR$, 
\[
  s^2 \leq \fourth b_2 s^4 - \half b_1 s^2 + C_{b_1,b_2} \txtwith C_{b_1,b_2}\dfn - \frac{\max\{0,b_1+2\}^2}{4 b_2}.
\]
In particular,  $\|w_x(t)\|^2 \leq  \cE(t) + C_{b_1,b_2}$ for any $t\geq 0$. Fix any $M>\max\{\sqrt{\cE(0) + C_{b_1,b_2}},\; \sqrt{|b_1|} \}$  (the $\sqrt{|b_1|}$ constraint is not used right here and instead comes from the previous considerations in \eqref{def:L}). Then  then $\|\p_x w_0\| < M$. By continuity of the trajectory in $\scrH_\al$, the inequality persists for $w(t)$ on some maximal finite interval $[0,T_1]$. From the energy identity \eqref{ei} (recall that  for well-posedness analysis we reduced to the case $U=0$, $p_0=0$) we have 
\[ 
  \cE(T) + C_{b_1,b_2}\leq \cE(0)+C_{b_1,b_2}  \txtfor T\in [0,T_1].
\]
But then $\|\p_x w(T_1)\|<M$ contradicting the construction of $T_1$. Hence,
\[
  \|w_x\| \leq M\txtforall t\geq 0.
\]
In this case, the solution to the truncated system \eqref{y-M-formulation} coincides with a solution to the original system \eqref{y-formulation} for all $t\geq 0$. Conversely, any solution to \eqref{y-formulation} coincides with the unique solution to the $M$-system (with $M>\sqrt{|b_1|}$) on any  time interval $[0,T]$ on which $\|u_x\|$ is bounded above by $M$.

Finally, the existence result for weak and strong solutions of the forced system \eqref{Bergerplate} with forcing $\tl{p}(x,t)=p(x,t) - (k_0+\bet) w_t(x,t)$, readily follows  when either  $p\in W^{1,1}_{loc}([0,\infty);L^2(0,L))$, or when $p(x,t) = L(x, w(x,t), w_t(x,t))$ where operator $L$ is Lipschitz $\scrH_\al\to L_\al$, in the sense  that 
\[
  \| L(\cdot, w,v)-L(\cdot, \tl{w},\tl{v})\|_{L_\al}\leq C( \|v-\tl{v}\|_{L_\al} + \|w - \tl{w}\|_{H^2(0,L)}),
\]
see  \cite[Prop. 4.1 \& Coro. 4.1, pp. 180--181]{b:showalter:97}.

 \section{Long-time Behavior}
This section is divided into subsections concerning the dissipativity and asymptotic smoothness (and thus asymptotic compactness) for the dynamical system $(\mathscr H_{\alpha},\mathcal S(t))$, $\alpha>0$. Recall that for $\alpha=0$, we do not verifiably have a dynamical system to analyze. However, since weak solutions exist by Theorem \ref{thm:weak}, we do make some comments below in the case of $\alpha=0$. Regardless of the value of $\alpha \ge 0$, we emphasize that  the boundedness of trajectories does not follow immediately from the energy identity \eqref{ei}. The key feature here is that the nonlinear energy provides control of the energy building ``piston" terms in the energy identity.

 \subsection{$\alpha>0$: Absorbing ball}
We now prove Proposition \ref{ball}---the case of rotational inertia. 

\begin{proof}[Proof of Proposition \ref{ball}]
Let us start with a couple of useful consequences of the energy identity \eqref{ei}. Using the definition \eqref{def:energy2} of $\cl{E}$ and   Proposition \ref{prop:compare} we have, via Gr\"onwall's estimate
\begin{equation}\label{exp-bound}
  \cl{E}(t) \leq  \cl{E}(s) \exp^{\om T} \txtforall t\in [s,s+T],\; s\geq 0.
 \end{equation}
 Since the system is autonomous, then by merely relabeling the start and end times we can apply \eqref{ei} on interval $[t,T]$.  
 
For convenience, let's introduce
\[
  k\dfn k_0+\bet \txtand \tl{\bet} \dfn-\bet U\,.
\]
Then the energy identity \eqref{ei} reads 
\[
  \cE(T) +\al k_0 \int_t^T \|w_{tx}(s)\|^2 ds +  k \int_t^T  \|w_t(s)\|^2 ds  =\cE(t)+\int_t^T (p_0+ \tl{\bet} w_x(s), w_t(s))ds \txtfor T\geq t.
\]
Integrate with respect to $t\in[0,T]$, to we arrive at 

\begin{equation}\label{TET}
  \intT \cE(t) dt =  T \cE(T) +  \intT \int_t^T (k\|w_t(s)\|^2 + \al k_0 \|w_{tx}\|^2) ds dt - \intT \int_t^T (p_0+ \tl{\bet} w_x(s), w_t(s)) ds dt.
 \end{equation}

To take advantage of the above estimates we will also need multiplier identities. As usual, the multiplier calculations are applied to strong solutions end extended to weak ones by the density of the domain of the generator and the continuous dependence on the data in the finite energy space $\scrH_{\alpha}$. Using multiplier $w$ on \eqref{Bergerplate} gives
  \[
    \begin{split}
      \intT (\|w_t\|^2+\al  \|w_{tx}\|^2)  =& \intT (D\|w_{xx}\|^2 + b_2\|w_x\|^4 - b_1 \|w_x\|^2) + \intT (kw_t+\al k_0 w_{tx} - p_0 - \tl{\bet}w_x, w)\\
    &+ \bigg[ (w_t,w) + \al (w_{tx},w_x) \bigg]_0^T
\end{split}
\] 
Hence for $0< c \leq \min\{k,k_0\}/2$ 
\begin{equation}\label{equipart}
    \begin{split}
      & c\intT (D\|w_{xx}\|^2 + b_2\|w_x\|^4 - b_1 \|w_x\|^2)\\
      \leq & \intT \left( \frac{k}{2}\|w_t\|^2+\frac{\al k_0}{2}\|w_{tx}\|^2 \right)\\ 
      & -c\intT (kw_t+\al k_0 w_{tx} - p_0 - \tl{\bet} w_x, w)
    - c\bigg[ (w_t,w) + \al (w_{tx},w_x) \bigg]_0^T
\end{split}
\end{equation} 
Next, invoke the energy identity \eqref{ei} and combine it with \eqref{equipart} 
\[
  \begin{split}
    &  \cE(T) + \frac{k}{2} \intT\|w_t\|^2 + \frac{\al k_0}{2}\intT\|w_{tx}\|^2 + c \intT(D\|w_{xx}\|^2 + b_2\|w_x\|^4-b_1\|w_x\|^2)  \\
    \leq\; & \cE(0) + \intT (p_0+\tl{\bet}w_x, w_t) - 
    c\intT (kw_t +\al k_0 w_{tx} - p_0-\tl{\bet}w_x,w)  - c \bigg[ (w_t,w) + \al (w_{tx},w_x) \bigg]_0^T.
\end{split}
\]
The terms on the left dominate a multiple of the integral of the full energy $\cE$,  and the right-most term can be estimated via \eqref{energy-compare2}:
\[
  \begin{split}
    \cE(T) +  c_1 \intT \cE(t) dt \leq& \cE(0) + \intT (p_0+\tl{\bet}w_x, w_t) \\
    &- c\intT  (k w_t + \al k_0 w_{tx} -p_0- \tl{\bet} w_x, w)
   + C (\cE(0) + \cE(T))+C_2
  \end{split}
\]
Apply \eqref{TET} to rewrite half of $c_1\intT \cE(t) dt$ on the left, and use \eqref{ei} to rewrite $\cE(T)$ on the right:
\[ 
  \begin{split}
    &  \cE(T) + \frac{c_1}{2}\left( T \cE(T) +\intT \int_t^T (k \|w_t(s)\|^2 +\al k_0\|w_{tx}\|^2) - \intT \int_t^T (p_0+ \tl{\bet} w_x, w_t)  \right) +  \frac{c_1}{2}\intT \cE(t) \\
    \leq & \cE(0) + \intT (p_0+\tl{\bet}w_x, w_t) - c\intT  (k w_t + \al k_0 w_{tx} -p_0- \tl{\bet} w_x, w)
   +C\cE(0)\\
   &+  C\left( \cE(0) -  \intT k\|w_t\|^2 - \intT \al k_0 \|w_{tx}\|^2 + \intT (p_0+ \tl{\bet}w_x, w_t) \right) + C_2
 \end{split}
\]
Drop the (non-negative) double integral of $k\|w_t(s)\|^2+\al k_0\|w_{tx}\|^2$ on the left, and move $\ds \intT \int_t^T (p_0+ a w_x, w_t) $ to the RHS. The Cauchy-Schwartz inequality yields,  for any $\eps>0$, 
\[ 
  \begin{split}
     & \cE(T) + \frac{c_1 T}{2} \cE(T)  +  c_1\intT \cE(t) \\
     \leq & (1+2C)\cE(0) + C_2 +  \eps \intT (\|w_t\|^2+\|w_{tx}\|^2) + C_3 \eps^{-1}(1+T) \intT (\|p_0\|^2+ \|w_x\|^2 +\|w\|^2).
  \end{split}
\]
If we rewrite the integral of the energy on the left using the estimates of Proposition \ref{prop:compare} (in particular, $\const + \cE(t) \geq \half \cl{E}(t)$) we arrive at
\[ 
  \begin{split}
    &    \cE(T) + \frac{c_1 T}{2} \cE(T)  +  \frac{c_1}{4}\intT (\|w_t\|^2+ \al \|w_{tx}\|^2+ D \|w_{xx}\|^2 + b_2\|w_x\|^4)  -  c_1 T C_{b_1,b_2}\\
    \leq & (1+2C)\cE(0) + C_2 +  \eps \intT (\|w_t\|^2+\|w_{tx}\|^2) + C_3 \eps^{-1}(1+T) \intT (\|p_0\|^2+ \|w_x\|^2+\|w\|^2),
    \end{split}
\]
wherein for the choice $\eps \leq (c_1/4)\min\{1,\al\}$, the kinetic term on the RHS can be absorbed into the left:
\[
  \begin{split}
&  \cE(T) + \frac{c_1 T}{2} \cE(T)    +  \frac{c_1}{4}\intT (D \|w_{xx}\|^2+b_2\|w_x\|^4)\\
  \leq & (1+2C)\cE(0) + C_2 + C_3 \eps^{-1}(1+T) \intT( \|p_0\|^2+ \|w_x\|^2+\|w\|^2)+c_1 T C_{b_1,b_2}.
  \end{split}
\]
Note that,  for any $\del>0$, there exists a constant $K_{p_0,\del,\eps,T}$ independent of the solution, such that
\[
  C_3 \eps^{-1}(1+T)(\|p_0\|^2 + \|w_x\|^2 + \|w\|^2) \leq  \del \|w_x\|^4  + K_{p_0,\del,\eps,T}.
\]
Thus, for $\del < c_1/4$, and abbreviating $K_{p_0, \del,\eps,T}$ as $K_T$,  arrive at 
 \[
   \cE(T) + \frac{c_1 T}{2}  \cE(T)  \leq (1+2C) \cE(0)  +  T K_{T}.
\]
At this point let's switch to the positive definite energy $\cl{E}$,  \eqref{def:energy2}. By means of Proposition \ref{prop:compare}, infer:
\[
  \cl{E}(T) + c_2 T \cl{E}(T) \leq c_3 \cl{E}(0)  + \cK_T
\]
for positive constants $c_1, c_2$ and $\cK_T$  independent of the solution itself. Divide by $c_2 T$ with $T$ large enough to provide $\sig \dfn \frac{c_3}{1+c_2 T}<1$:
\[
   \cl{E}(T) \leq \sig \cl{E}(0) +  \cM_T\txtwith  0<\sig <1
\]
where, again, constant $\cM_T$  is independent of the solution. Consequently,
\[
   \cl{E}(nT) \leq \sig^n \cl{E}(0) + \cM_T \sum_{j=0}^{n-1} \sig^j.
\]
In particular, for $\ds n>n_0\dfn -\frac{\ln\big(\max\{1,\cl{E}(0)\}\big)}{\ln(\sig)} $ (non-negative since $\sig\in(0,1)$), we have
\[
  \cl{E}(nT) \leq  1 + \cM_T  \frac{1}{1-\sig}.
\]
Any $t>0$ can be written as $t = m T + \tau$ for $m>n_0$ and  $\tau \in [0,T)$. Then via \eqref{exp-bound}
  \[
   \cl{E}(t) \leq  c \cl{E}(mT) e^{\om T} \leq  c e^{\om T} \left( 1+ \frac{\cM_T}{1-\sig} \right)
 \]
for $t$ bigger than some $t_0$ dependent on $\cl{E}(0)$. Since $\cl{E}(0)$ is coercive with respect to the $\scrH_\al$-norm of the initial data, then this inequality confirms the existence of an absorbing ball.

\end{proof}
\begin{remark}[Conjecture for $\al=0$]\label{alf}
  The same proof as above {\em formally} carries through when $\al=0$ (mutatis mutandis), thus suggesting a bounded absorbing set for the non-rotational model. However, with $\al=0$, it is not known if the solutions are unique, hence that model may have to be considered in the context of generalized semi-flows (e.g., \cite{ball:87:JNS}).  Also the energy and the equipartition identities cannot be cited from the semigroup version, which admits them for {\em strong solutions} and extends by density. It appears plausible, nonetheless, that the non-rotational system shares long-term behavior features with the rotational one. We will examine the case $\al=0$ numerically in Section \ref{numerics} and, in fact, provide some numerical evidence of boundedness of solutions in the non-rotational case.
\end{remark}

\subsection{$\alpha>0$: Asymptotic compactness}
For $\alpha>0$, we show the {\em asymptotic smoothness} property via Theorem \ref{psi}. 
With dissipativity established, once asymptotic compactness of $(\mathscr H_{\alpha},\mathcal S(t))$ is shown, Theorem \ref{dissmooth} guarantees the existence of a compact global attractor. (This is the approach taken for obtaining global attractors in \cite{beam3,beam4} for the cantilevered extensible beam taken with boundary damping and sources.)

For brevity, in the subsequent estimates we will occasionally use the following notation:
\[
  \|w\|_\th \dfn \|w\|_{H^\th(0,L)}\,.
\]
In particular, for any solution $w$, $\|w_t\|_1$ is equivalent to $\|w_t\|_{H^1_*}$ and  
$\|w\|_2\equiv \|w\|_{H^2_*}$.

To begin, we consider some invariant (with respect to $\mathcal S(t)$), bounded set $\mathscr B \subset \mathscr H_{\alpha}$. There exists $R$ such that $\mathscr B\subset B_R(\mathscr H_{\alpha})$. With the existence of the absorbing ball for $(\mathscr H_{\alpha},\mathcal S(t))$ in place, we know there is a time $T_R$ so that $\mathscr B$ is absorbed by the absorbing ball. By the continuity of the energy functional (or by Gr\"{o}nwall) in Theorem \ref{thm:semigroup}, and the equivalence described in Proposition \ref{prop:compare} on $[0,T_R]$, we may restrict our attention to
$w^1,w^2 \in \mathscr B$ satisfying
\begin{equation*}
\|w^1(t)\|_{2}+\|w^1_{t}(t)\|_{1}+\|w^2(t)\|_{2}+\|w^2_{t}(t)\|_{1}\leq
C(R),\quad t>0. \end{equation*}
This assumption on the boundedness of trajectories will stand throughout this section.

As one can see by Theorem \ref{psi}, the key to asymptotic compactness relies on estimating the difference to two trajectories. To do so, one considers the following difference system, where $w^i$ each satisfy \eqref{Bergerplate}, $z=w^1-w^2$ and we write $B(w) = (b_1-b_2\|w_x\|^2)$:
\begin{equation}\label{difference}
\begin{cases}
(1-\alpha\partial_x^2)z_{tt}+D\partial_x^4z-\alpha k_0\partial_x^2z_t+kz_t+ \big(B(w^1)[w^1]_{xx}-B(w^2)[w^2]_{xx}\big) = -\beta Uz_x \\
z(0)=z_x(0)= 0 \\
z_{xx}(L)=0;~~-\alpha\partial_x[z_{tt}(L)+k_0z_t(L)]+D\partial_x^3z(L)+B(w^1)[w^1]_x(L)-B(w^2)[w^2]_x(L)=0\\
z(0)=w^1_0-w^2_0;~~z_t(0)=w^1_1-w^2_1.
\end{cases}
\end{equation}
(Note: as above, we are taking $k=k_0+\beta$, with $\beta>0$.) We let $\cF(z) =B(w^1)[w^1]_{xx}-B(w^2)[w^2]_{xx}$. We also introduce a modified energy for the $z$ trajectories:
\begin{equation}
E_z(t) :=\dfrac{1}{2}\Big\{D\|z_{xx}\|^2+ \|z_t\|_{L_{\alpha}}^2\Big\}.
\end{equation}
As is typical, the key term to be estimated for long-time behavior analysis is $\ds \int_s^t\left(\cF(z),z_t\right)d\tau.$ 
This term, however, is not compact in any obvious way.
In line with previous analyses \cite{Memoires,hlw,beam4}, 
\begin{align}\nonumber
\cF(z)=&~(b_1-b_2\|w^1_x\|^2)[w^1]_{xx}-(b_1-b_2\|w^2_x\|^2)[w^2]_{xx}\\
=&~B(w^1)z_{xx}+\left[B(w^1)-B(w^2)\right][w^2]_{xx}.\label{nondiff}
\end{align}
 \begin{remark} Unlike analyses involving clamped or hinged boundary conditions, this decomposition is more challenging  because of the nonlinear boundary conditions; indeed, the inextricable link between the elasticity operator, the rotational terms, and the nonlinearity (through the boundary condition), make estimation of differences challenging.  This is why the nonlinearity cannot be treated as a perturbation of the linear dynamics, as is typically the case with $f$ coming from the theory of large deflections \cite{lagnese,springer}.
 In fact, for $f(w)=(b_1-b_2\|w_x\|^2)w_{xx}$, when $w \in H_0^1(0,L)\cap H^2(0,L)$ (and hence in the context of fully clamped or hinged beams), we see that $f$ is locally Lipschitz into $L^2(0,L)$.  \end{remark}

Proceeding to  analyze the key term, we see that we must consider the elastic and inertial components in conjunction with the nonlinearity due to the boundary conditions. We first provide an identity that will be used for the equipartition multiplier $z$ and the energy multiplier $z_t$.

\begin{lemma} For $(z,z_t,z_{tt})$ a strong solution to \eqref{difference} and $\phi \in H^2_*$ we have the identity:
\begin{equation} 
  \begin{split}
    \big(-\alpha\partial_{x}^2(z_{tt}+k_0z_t)+& D\partial_x^4z+\mathcal F(z),\phi\big) \\
=&~\Big(\alpha\partial_x(z_{tt}+k_0z_t)-D\partial_x^3 z-B(w^1)z_{x}-\left[B(w^1)-B(w^2)\right][w^2]_{x}, \phi_{x}\Big)\\
=&~\alpha(\partial_xz_{tt},\phi_x)+\alpha  k_0(\partial_xz_t,\phi_x)+D(z_{xx},\phi_{xx})-B(w^1)(z_x,\phi_x) \\
& -\left[B(w^1)-B(w^2)\right]([w^2]_{x},\phi_x) \label{multiplier}
\end{split}
\end{equation}
\end{lemma}
\begin{proof} We integrate by parts in the relation \eqref{nondiff} and invoke the boundary conditions, as in \eqref{difference}.
\end{proof}
In what follows we will use $\phi=z_t$ and $\phi=z$ as multipliers (again, first with strong solutions, then on generalized solutions via density). This will result in our key observability inequality, and allows us to invoke Theorem \ref{psi}.

\begin{lemma}\label{observe}
Let $T>0$. Generalized solutions on $[0,T]$ to \eqref{difference} satisfy the inequality
\begin{equation}\label{observeinequ}
TE_z(T)+\int_0^TE_z(\tau)d\tau \le C(R)E_z(0)+C(R,T,b_2)l.o.t._{[0,T]},
\end{equation}
where $\displaystyle l.o.t._{[0,T]} := \sup_{[0,T]}\|z\|_{2-\eta}^2$.
\end{lemma}
\begin{proof} The proof is standard (see \cite[Lemma 8.3.1, p.381]{springer} for an abstract version for second order problems, or the discussion of Lemma 3.1 in \cite{hlw}), and follows along the lines of the proof of Proposition \ref{ball} utilizing the energy and equipartition multipliers. Here, we focus only the key issue: the coupling of the nonlinearity and boundary conditions, and addressing the standard ``trick" given in \eqref{nondiff} of tackling nonlinear differences.
First, we take $\phi=z_t$ in \eqref{multiplier} to arrive at:
\begin{align} \nonumber \big(-\alpha\partial_{x}^2(z_{tt}+k_0z_t)+D\partial_x^4z+\mathcal F(z),z_t\big)& \\=~\alpha(\partial_xz_{tt},\partial_xz_t)+\alpha & k_0(\partial_xz_t,\partial_xz_t)+D(z_{xx},z_{txx})-B(w^1)(z_x,\partial_xz_t) \\\nonumber
-\left[B(w^1)-B(w^2)\right]&([w^2]_{x},\partial_xz_t)\end{align} 
The first few terms can be rewritten as total derivatives or damping in the individual inner products, leaving two remaining terms to be addressed:
\begin{align*} (-\alpha \partial_x^2(z_{tt}+k_0z_t)+\partial_x^4z+\mathcal F(z),z_t) =&~\dfrac{1}{2}\dfrac{d}{dt}\Big(\alpha \|\partial_xz_t\|^2+D\|z_{xx}\|^2-B(w^1)\|z_x\|^2\Big)+k_0\alpha\|\partial_xz_t\|^2\\&+\frac{1}{2}\dfrac{d}{dt}\left(B(w^1)\right)\|z_x\|^2-\left[B(w^1)-B(w^2)\right]([w^2]_{x} , z_{xt}).\end{align*}

We rewrite the last two terms, using integrations by parts, with a mind to estimate them:
\begin{align}\label{bad}
\dfrac{1}{2}\dfrac{d}{dt}(B(w^1)) \|z_x\|^2= &-b_2(w^1_x,w^1_{xt})\|z_x\|^2 = \Big(b_2(w^1_{xx},w^1_t)-b_2[w^1_x(L)w^1_t(L)]\Big)\|z_x\|^2\\\label{bad*}
\left[B(w^1)-B(w^2)\right]([w^2]_{x}\big], z_{xt}) = &-\left[B(w^1)-B(w^2)\right] ([w^2]_{xx},z_t)+\left[B(w^1)-B(w^2)\right]\left[[w^2]_x(L)z_t(L)\right]
\end{align}
Then, utilizing the above and reorganizing, we have
\begin{equation}\label{operateon}
  \begin{split}
&\dfrac{1}{2}\Big[D\|z_{xx}(T)\|^2+\alpha\|\partial_xz_t(T)\|^2\Big]+\alpha k_0\int_0^T\|\partial_xz_t\|^2d\tau \\
\le&~ \frac{1}{2}E_z(0) +\Big[|B(w^1(t))|\cdot\|z_x(t)\|^2\Big]\Big|_0^T\\
&+\Big|\int_0^Tb_2(w^1_{xx},w^1_t)\|z_x\|^2d\tau\Big|+\Big|\int_0^Tb_2[w^1_x(L)w^1_t(L)]\|z_x\|^2d\tau \Big|  \\ 
&+\Big|\int_0^T\left[B(w^1)-B(w^2)\right] ([w^2]_{xx},z_t) d\tau \Big| 
+\Big| \int_0^T\left[B(w^1)-B(w^2)\right]\left[[w^2]_x(L)z_t(L)\right]  d\tau \Big|
\end{split}
\end{equation}
The terms at $x=L$ may be estimated using the trace theorem and interpolation directly, yielding:
\begin{align}\label{operateon*}
\dfrac{1}{2}\Big[D\|z_{xx}(T)\|^2+\alpha\|\partial_xz_t(T)\|^2\Big]+\alpha k_0\int_0^T\|\partial_xz_t\|^2d\tau \le&~ \frac{1}{2}E_z(0) +C(R)l.o.t._{[0,T]}\\
&+C(R,T)l.o.t._{[0,T]}\nonumber\\
&+b_2\int_0^T\|w\|_{3/2+\delta}\cdot \|w^1_t\|_{1/2+\delta}\cdot\|z_x\|^2d\tau \nonumber \\ \nonumber
&+\int_0^T\Big|B(w^1)-B(w^2)\Big|\cdot \|w^2\|_2\cdot \|z_t\|_0 d\tau \\
&+\int_0^T\Big|B(w^1)-B(w^2)\Big|\cdot  \|w^2\|_{\frac{3}{2}+\delta}\cdot \|z_t\|_{\half+\delta}  d\tau. \nonumber
\end{align}
We will now invoke the assumption that $\|w^i\|_2^2+\|w^i_t\|_1 \le C(R)$. For the term $B(w^1)-B(w^2)$, we have
\[
  |B(w^2)-B(w^2)| \le b_2\big|\|w^1_x\|^2-\|w^2_x\|^2\big| \le C(b_2)(\|w^1\|_1+\|w^2\|_1)\,\big|\,\|w^1\|_1-\|w^2\|_1\,\big|,
\]
from which we see that \begin{equation}\label{bdiff} |B(w^2)-B(w^2)| \le  C(b_2)\big(\|w^1\|_1+\|w^2\|_1\big)\|z\|_1 \le C(R,b_2)\|z\|_1.\end{equation}
We can thus estimate each line on the RHS of \eqref{operateon*} respectively as follows:

\begin{align}\label{operateon**}
\dfrac{1}{2}\Big[D\|z_{xx}(T)\|^2+\alpha\|\partial_xz_t(T)\|^2\Big]+\alpha k_0\int_0^T\|\partial_xz_t\|^2d\tau \le&~ \frac{1}{2}E_z(0) +C(R)l.o.t._{[0,T]}\\
&+C(R,T)l.o.t._{[0,T]}\nonumber\\
& +C(b_2,R,T)l.o.t._{[0,T]}\nonumber\\
& +C(b_2,R,T)l.o.t._{[0,T]}\nonumber\\
&+\epsilon \int_0^T \|\partial_x z_t\|^2d\tau +C(\epsilon,R,b_2,T)l.o.t._{[0,T]}.\nonumber
\end{align}
Finally, this results in the estimate
$$\dfrac{1}{2}\Big[D\|z_{xx}(T)\|^2+\alpha\|\partial_xz_t(T)\|^2\Big]+\alpha k_0\int_0^T\|\partial_xz_t\|^2d\tau \le \epsilon \int_0^T\|\partial_x z_t\|^2d\tau+C(R,T,b_2,\epsilon)l.o.t._{[0,T]}, ~~\epsilon>0.$$
Taking $\phi=z$ in \eqref{multiplier}, the estimation proceeds more easily, as there is no need to handle time derivatives as in \eqref{bad}--\eqref{bad*}.
 Thus the inequality in Lemma \ref{observe} is obtained from the multipliers $z_t$ and $z$ on \eqref{difference} using the estimations above.
\end{proof}

Now, we may take $\Psi\equiv l.o.t._{[0,T]}$. To apply Theorem \ref{psi}, we simply let $T$ be sufficiently large in \eqref{observeinequ} (relative to $R$) and note that $l.o.t._{[0,T]}$ are compact with respect to the space $\mathscr H_{\alpha}$, and thus the iterated limit property follows easily. Finally, invoking Theorem \ref{psi}, we obtain that $(\mathscr H_{\alpha},\mathcal S(t))$ is asymptotically smooth. This step in turn concludes the proof of Theorem \ref{attractor}.
\begin{remark}
This last step is much simpler than in the case of von K\'arm\'an dynamics for a plate, where the structure of $\Psi$ has more terms dependent upon the structure of the nonlinearity. In that case, one must forgo compactness and exploit the strength of the iterated compensated compactness allowed by Theorem \ref{psi}. See \cite[Section 8.3.2]{springer}.
\end{remark}

\begin{remark}[Lack of trace regularity]\label{alphazerosmooth}
The key issue above is precisely the terms with boundary traces:
\begin{align}
b_2[w^1_x(L)w^1_t(L)]\|z_x\|^2;~~\hskip1cm
\left[B(w^1)-B(w^2)\right]\left[[w^2]_x(L)z_t(L)\right].
\end{align}
These terms are present regardless of the value of $\alpha \ge 0$. It becomes clear, here, that to control these terms one must have control of $z_t(L)$, which is not a priori defined in the case $\alpha =0$ (with $z_t \in L^2(0,L)$ only). Thus, the standard tack of estimating 
$(\cF(z),z_t)$ via the decomposition approach above hinges upon this point. Moreover, we emphasize that---owing to the lack of the local Lipschitz property of $f$ here---writing 
$\ds \left|(\cF,z_t)\right| \le \|\cF(z)\|\|z_t\|$ and imposing large interior damping ($k_0 \gg 0$) when $\alpha=0$ is not sufficient to obtain the estimates. These hurdles present themselves in an analysis depending on the difference of two trajectories, including uniqueness of $\alpha=0$ weak solutions from Theorem \ref{thm:weak}. 

One {\em can} control these terms with boundary damping of the form $g(z_t)$ at $x=L$ imposed in the higher order condition. This is precisely the approach taken by \cite{beam1,beam3,beam4}.
\end{remark}

\section{Numerical Simulations}\label{numerics}
\def\Ucrit{U_{\text{crit}}}
\def\hUcrit{\hat{U}_{\text{crit}}}
\def\emax{\mathcal{E}_{\text{max}}}
In this section we consider dynamics of the form
\begin{equation}\label{k1}(1-\alpha\partial_x^2)w_{tt}+D\partial_x^4 w+ {k_0w_t-k_1\partial_x^2w_t}+(b_1-b_2\|w_x\|^2)w_{xx} = p_0(x)-\beta(w_t+Uw_x),\end{equation} with the associated nonlinear cantilevered boundary conditions (analogous to those in \eqref{Bergerplate}) with specific parameter choices. 
Note that we have decoupled the ``viscous" damping from the ``strong" damping treating $k_0$ and $k_1$ as fully independent parameters (unlike the theoretical analysis above where we took $k_1 =\alpha k_0$).  The piston-theoretic right hand side is determined by three quantities: $\beta$, representing the scaling of the fluid downwash (that is, due to flow effects), $p_0$, representing a static fluid pressure, and $U$, the unperturbed flow velocity. For convenience we fix $\beta$ at unity, and we are not primarily interested in the effect of the function $p_0$ on the long-time behavior of trajectories.
Thus, for the numerical study below we take
\[
  \bet=1,~ \txtand  p_0 \equiv 0.
\] 
We also note that a non-zero $b_1$---a pre-stressing parameter---is typically associated with configurations that restrict both beam ends. We included nonzero $b_1$ for generality in our theoretical work, but we do not focus on it in our simulations below, and thus we take
\[b_1=0.\]

To demonstrate several qualitative aspects of dynamics considered here,  we conduct numerical simulations on the piston-theoretic, cantilevered extensible beam model driven by \eqref{k1}.  Spatial discretization on the first-order evolution system is accomplished via a finite element method, and temporal integration is performed using the Runge-Kutta method of fourth order.  All simulations were performed with {\em mathematical} parameter choices\footnote{One can find appropriate physical scalings, for instance, in \cite{vedeneev}.}: $$D=1,~~~L=1,$$ and a spatial mesh size of $\Delta x=\ell/20$. {\em Unless stated otherwise, initial data for all simulations included an equilibrium initial displacement and linear initial velocity:  
\[
w(x,0)=0,\qquad w_t(x,0)=0.01x.
\]}
In general, meaningful observations of the beam dynamics could be made in the range $0\le t\le 20$, so the running time is $T=20$ for most simulations below.

In what follows, the central focus on stability revolves around the piston-theoretic ``perturbation" term ~$-U w_x$ in the equation, and its effect on stability properties of the dynamics. For various parameter combinations, we will determine a so called $U_{\text{crit}}$ which corresponds to the onset of flow-induced instability. To reiterate, when we are considering a $U<U_{\text{crit}}$ for a particular configuration, we expect and demonstrate that the dynamics (linear or nonlinear) converge---with exponential rate---to the equilibrium. In the supercritical case $U>U_{\text{crit}}$, we expect linear dynamics ($b_2=0$) to exhibit exponential growth of energies, and we expect nonlinear dynamics ($b_2>0$) to exhibit fluttering behavior---typically characterized by limit cycle oscillations (LCOs) resembling the second in vacuo cantilever beam eigenmode.
We now describe several different studies on the simulated dynamics.

\subsection{Basic qualitative properties of flutter}

To ground our discussions, we begin by showing a collection of snapshots of the in vacuo \emph{linear} ($b_2=0$) beam dynamics corresponding to the first and second Euler-Bernoulli cantilever mode shapes. The simulations below take us through one period of the beam dynamics (\eqref{k1} taken with $\alpha=k_0=k_1=b_1=b_2=p_0=\beta =0$ and $D=1$), with initial displacements taken as the first and second cantilever modes. 
\begin{figure}[htp]
\begin{center}
\includegraphics[width=3in]{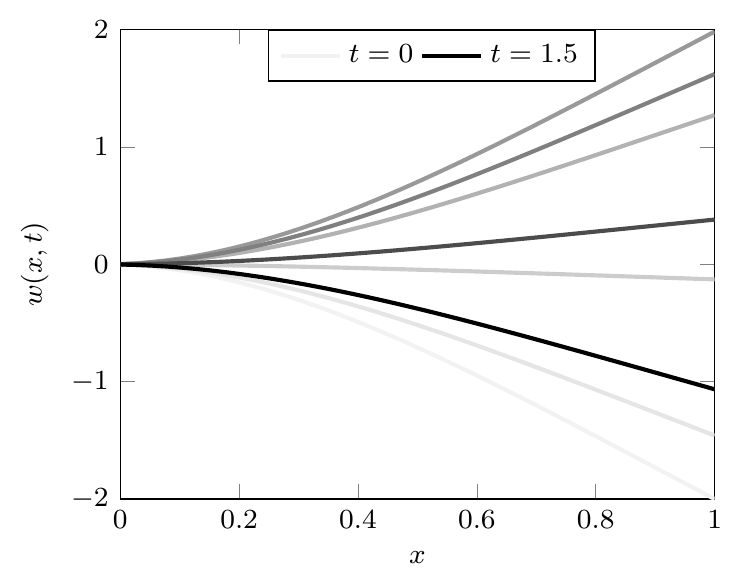}
\hspace*{0.25in}
\includegraphics[width=3in]{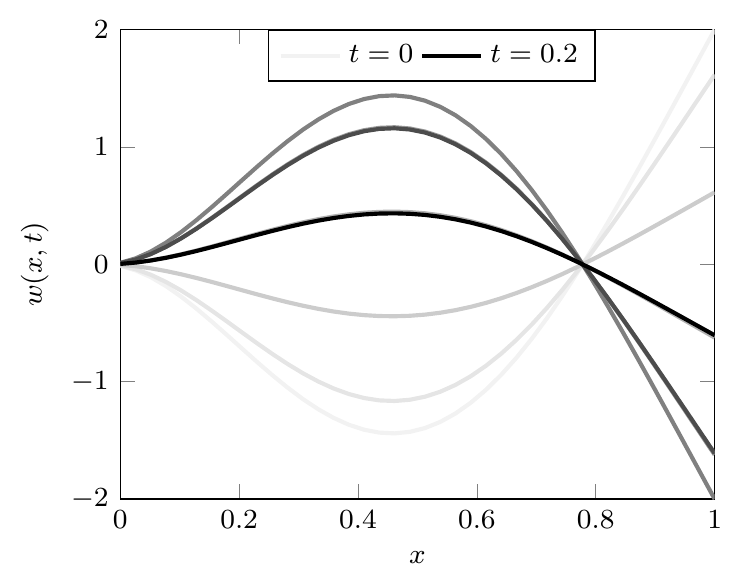}
\caption{Plot of in vacuo $w(x,t)$ at varying $t$; 1st mode as initial displacement (left), 2nd mode as initial displacement (right).}\label{fig21}
\end{center}
\end{figure}

We now provide a similar snapshot of the profile of a fluttering cantilever beam ($b_2=1$), taken with a polynomial initial displacement (see \textbf{Polynomial ID} below). We show approximately one period of the non-transient dynamics of a beam after it has approached a limit cycle. Note the similarity of the flutter profile and the second in vacuo mode profile. 
\begin{figure}[htp]
\begin{center}
\includegraphics[width=3in]{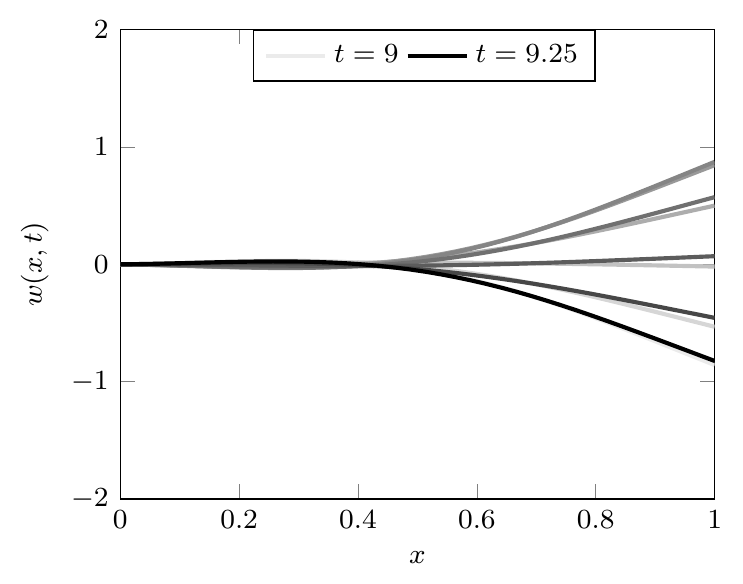}\hspace*{0.25in}
\includegraphics[width=3in]{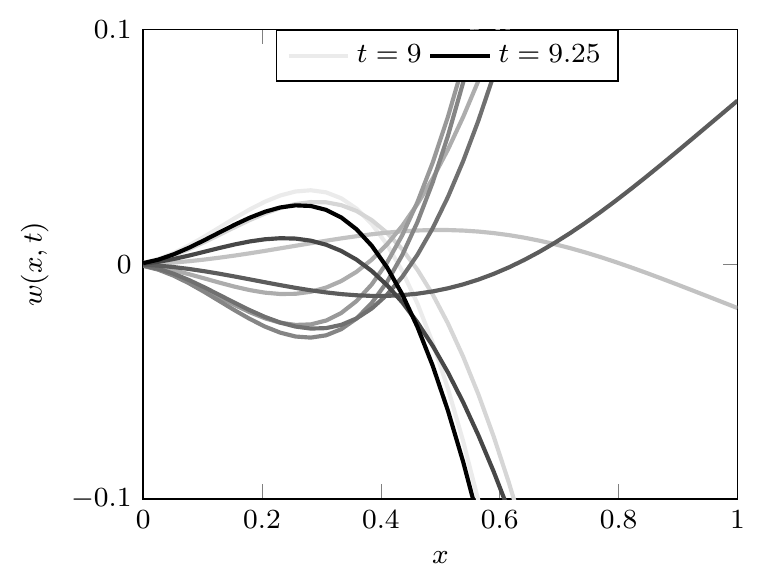}
\caption{Plot of $w(x,t)$ for $\alpha=k_1=k_0=0$, $b_2=1$, $U=150$ at varying $t$; {\bf Polynomial ID}. The figure on the right is a magnification in the transverse dimension of the left.}\label{fig29}
\end{center}
\end{figure}

Next, we consider the nonlinear case $(b_2=1)$ and provide tip profile---the displacement at $x=L=1$---for three different parameter sets, each for three different initial configurations:
\begin{itemize}
  \item \textbf{[2nd Mode ID]}
$w(0,x)=s_2(x) =  [\cos(\kappa_2x)-\cosh(\kappa_2x)]-\mathcal C_2[\sin(\kappa_2x)-\sinh(\kappa_2x)], ~w_t(0,x)=0$,  \vskip.05cm
where $\kappa_2\approx 4.6941$ is the second Euler-Bernoulli cantilevered mode number (with $L=1$) and
 
 $\mathcal C_2 = \left[\dfrac{\cos(\kappa_2)+\cosh(\kappa_2)}{\sin(\kappa_2)+\sinh(\kappa_2)}\right]\approx 1.0185$; 
 \item \textbf{[Polynomial ID]}\quad
$w(0,x)=-4x^5+15x^4-20x^3+10x^2$, ~$w_t(0,x)=0$; 
\item \textbf{[Linear IV]}\quad $w(0,x)=0$, $~w_t(0,x)=x$. 
\end{itemize} 

First, Figure \ref{fig26} represents the nonlinear ($b_2=1$), in vacuo ($\beta=k_0=k_1=0$), $\alpha=0$ tip displacements with the three initial configurations described above.
\begin{figure}[htp]
\begin{center}
\includegraphics[width=5in]{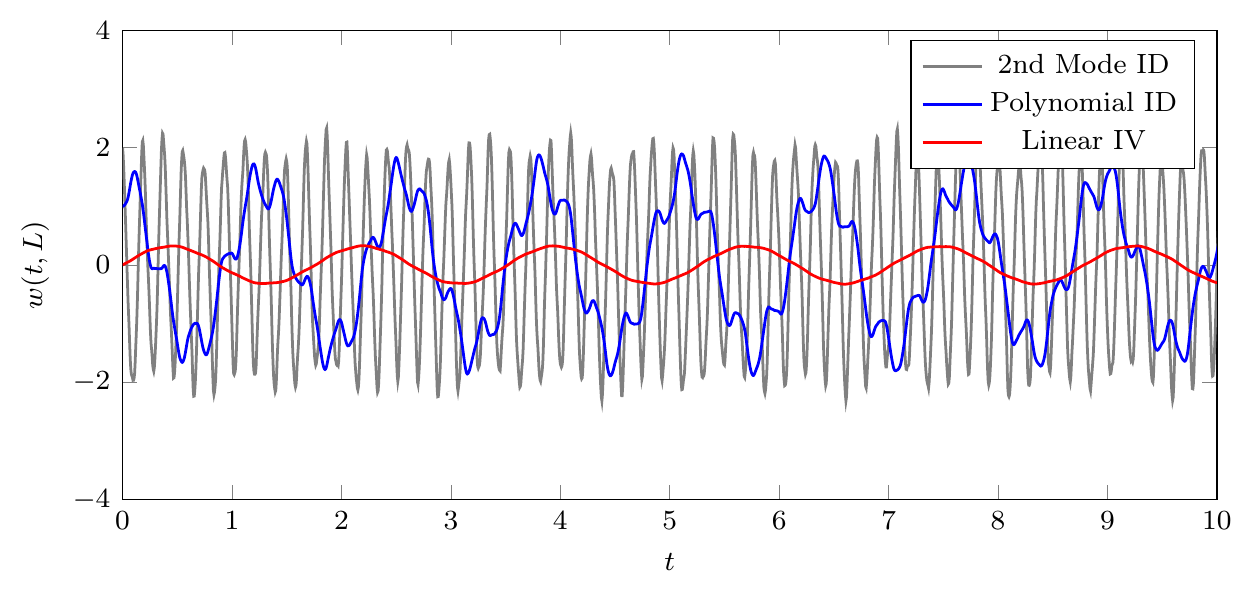}\\
\caption{Plot of in vacuo tip displacement $w(t,L)$; varying initial configuration.}\label{fig26}
\end{center}
\end{figure}
Figures \ref{fig23} and \ref{fig25} represent fluttering (nonlinear) dynamics in the case of $\al=0$, as well as in a rotational configuration ($\al=k_0=0.1$). 
Figure \ref{fig26} allows for comparisons of amplitudes and periods between the aforementioned nonlinear, in vacuo dynamics and the corresponding fluttering dynamics.
In the flutter cases, we see convergence to the ``same" LCO for each of the initial conditions considered. We also note the effects (on amplitude and period of the LCO) of the inherent parameters (e.g., $k_1$ and $\alpha$ here). Finally,  note that, although we seem to see convergence to the same LCO, the {\em transient regime} and ``time to convergence" are certainly affected by the choice of initial configuration. 
\begin{figure}[htp]
\begin{center}
\includegraphics[width=5in]{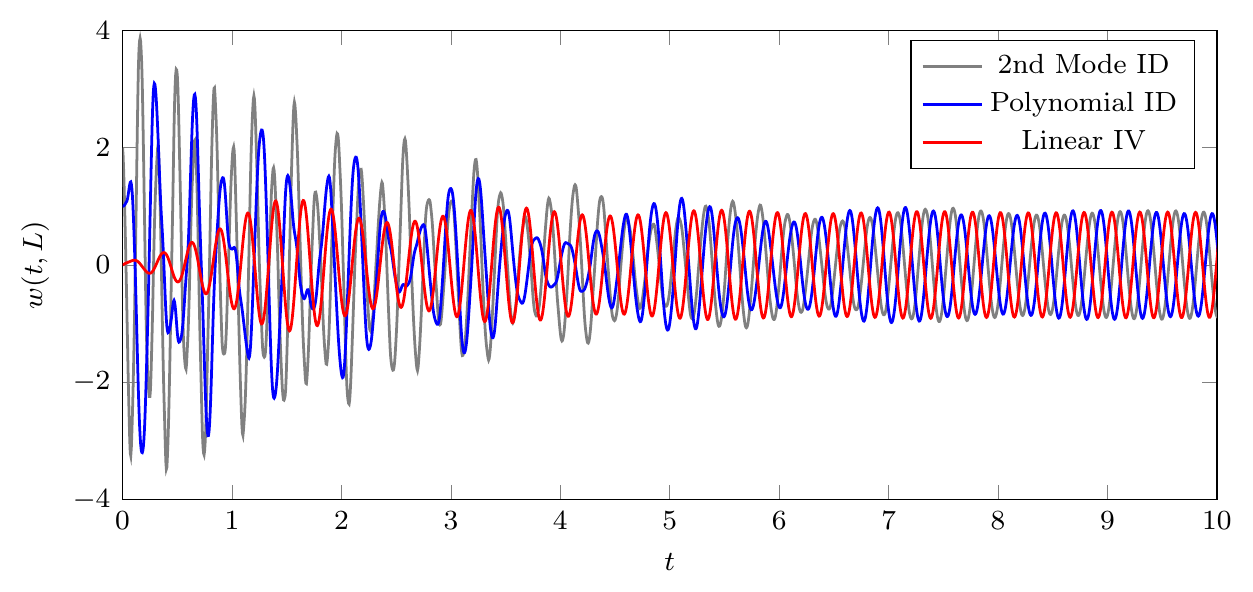}\\
\caption{Plot of $w(t,L)$ for $\alpha=k_1=k_0=0$, $b_2=1$, $U=150$; varying initial configuration.}\label{fig23}
\end{center}
\end{figure}
\begin{figure}[htp]
\begin{center}
\includegraphics[width=5in]{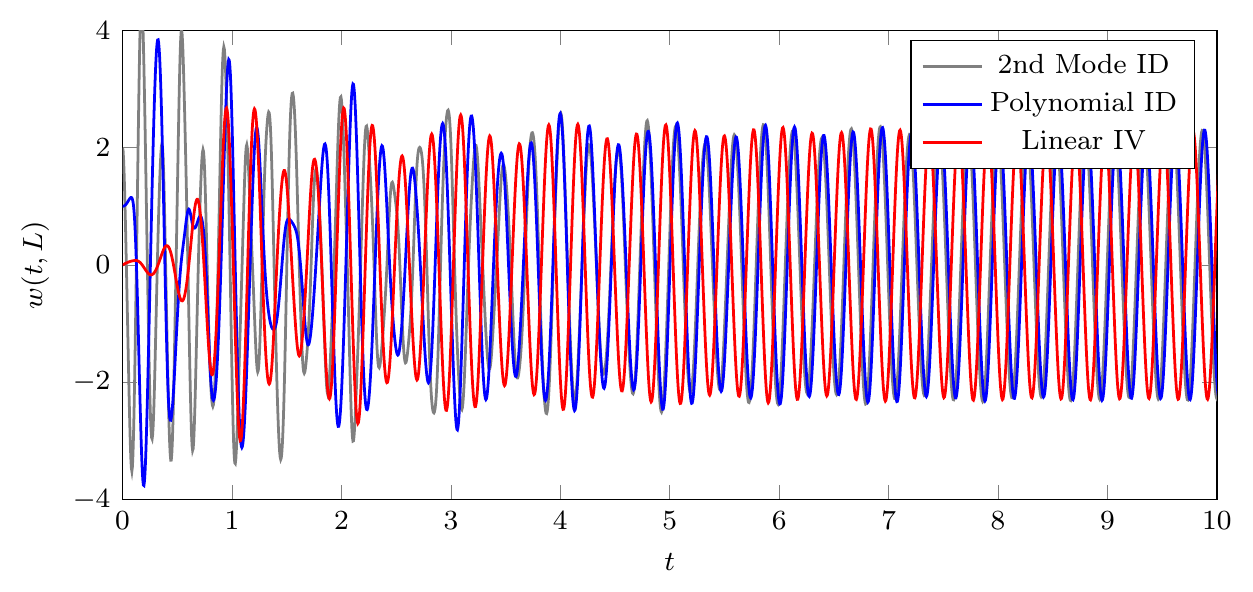}\\
\caption{Plot of $w(t,L)$ for $\alpha=k_1=.01$, $k_0=0$, $b_2=1$, $U=150$; varying initial configuration.}\label{fig25}
\end{center}
\end{figure}

\subsection{Influence of rotational inertia on critical flow velocities}
The next step in our study focuses on the influence of rotational inertia on the stability of the flow-perturbed dynamics.  In particular, for a linear model, we think of the flow term $Uw_x$ as a perturbation that can affect the location of the eigenvalues for the in vacuo beam. Without nonlinear effects included, an {\em unstable} configuration will exhibit exponential grow of energies in time, in accordance with the emergence of an unstable (perturbed) eigenvalue.\footnote{Another popular approach to assess the stability of the linear model is the so called ``modal" analysis, popular in the engineering literature (e.g., \cite{vedeneev}). This is a spectral approach that makes use of a Galerkin procedure with a basis given by the structure's in vacuo eigenfunctions. Treating lower order and damping terms as perturbations, and assuming simple harmonic motion in a dominant frequency, one reduces the linear stability problem to an eigenvalue computation. Our results here are easily checked via such an approach, though we refer to \cite{comppaper} for more details and a recent analysis of this type.}

Define the ``energy max'' $\emax(t_a, t_b)$ over the interval $[t_a, t_b]$ as
\[\emax(t_a, t_b) = \max_{t_a\le t\le t_b} \mathcal{E}(t).\]  
To empirically determine if a trajectory is stable or unstable here, we examine the ratio ~$r=\emax(20,40)/\emax(0,20)$ for a large class of simulations.  If $r>1$, then the trajectory will be considered unstable, and is stable otherwise.  In Figure \ref{fig0}, the rotational inertia parameter $\alpha$ is plotted against the critical flow velocity $U$ (here viewed as a function of  $\alpha \ge 0$), with the blue curve giving the approximate location of $\Ucrit(\alpha)$ for each value of $\alpha$. The damping coefficients $k_0$ and $k_1$ are taken to be $0$, and  nonlinear effects are disabled ($b_2=0$). We initialize with $\alpha=0$, with critical flow velocity $U_{\text{crit}}(0)=135.97$.  We then note that for $10^{-3}\le \alpha \le 10^{-1}$, the critical flow velocity drops precipitously. We also note that, as $\al \searrow 0$, the $U_{\text{crit}}(\alpha)$ appears to converge to $U_{\text{crit}}(0)=135.97$.
\begin{figure}[htp]
\begin{center}
\includegraphics[width=5in]{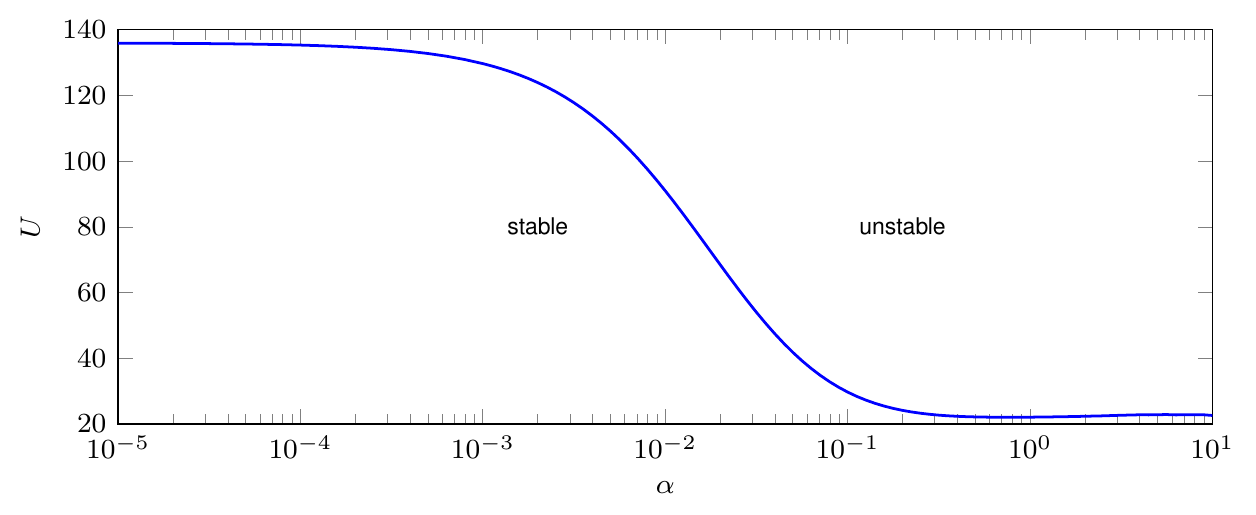}\\
\caption{Plot of $\mathcal{E}(t)$ for $k_1=k_0=0$, $b_2=0$, and varying $U$; $U_{\text{crit}}(\alpha=0)=135.97$.}\label{fig0}
\end{center}
\end{figure}
\begin{remark}
We include here, for reference, the comparative dispersion relations for the Euler-Bernoulli beam ($\alpha=0$) and the Rayleigh beam ($\alpha>0$) with $\alpha$ varying and $D=1$. Assuming a traveling wave solution of the form $w=Ce^{i(kx-\omega t)},~~C \in \mathbb C$, we have:
\begin{equation}
\omega^2 =\dfrac{k^4}{1+\alpha k^2} \txtfor  \alpha\ge 0.
\end{equation}
The above shows the relationship between wave number and eigenfrequency; in particular, if one invokes the boundary conditions of a particular configuration, one arrives at a relationship between the eigenvalues $\pm i\omega_n$ and $k_n$ (with $n \in \mathbb N$, $k_n \to +\infty$). The dispersion relationship demonstrates the {\em destabilizing} effect of $\alpha>0$, which reduces the magnitude of the in vacuo eigenvalues.
\end{remark}

\subsection{Influence of nonlinearity on boundedness of trajectories}
\subsubsection{$\alpha=0$: No rotational inertia}

First, computed energies $\mathcal{E}(t)$ (as given in \eqref{def:energy}) for the model \eqref{k1} with {$\alpha=0$}, no imposed damping ($k_0=k_1=0$), and with no nonlinear effects ($b_2=0$) are shown in Figure \ref{fig1}.  For this choice of parameters, an empirically determined approximation to the critical flow velocity is $\Ucrit(0)=135.97$.  Energy profiles (log-scale) for various choices of $U$ in terms of $\Ucrit$ are given---here, a linear profile in the energy plot represents exponential growth or decay of the dynamics in the energy norm, with margin of instability/stability depending on the slope of the profile.  Note that  the piston-theoretic damping induces exponential decay of energies $\cE(t)$  whenever $U<\Ucrit$, as evinced by the linear profiles (or envelopes) with a negative slope. On the other hand, when $U> \Ucrit$ the energy grows exponentially.
\begin{figure}[htp]
\begin{center}
\includegraphics[width=5in]{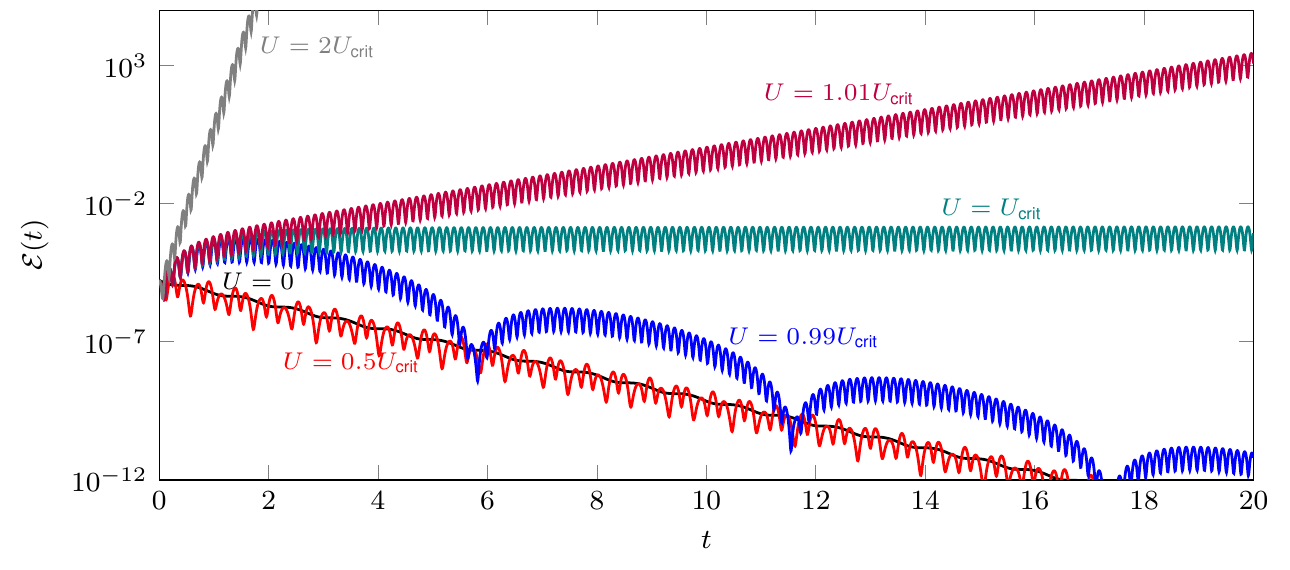}\\
\caption{Plot of $\mathcal{E}(t)$ for $\alpha=k_1=k_0=0$, $b_2=0$, and varying $U$; $U_{\text{crit}}=135.97$.}\label{fig1}
\end{center}
\end{figure}

In Figure \ref{fig2}, the nonlinear dynamics ($b_2=1$) are included into the model and the nonlinear energy $\mathcal{E}(t)$  is shown (for the same choices of $U$ as in Figure \ref{fig1}).  Note that $\mathcal{E}(t)$ (and hence $\cl{E}(t)$, due to \eqref{prop:compare}) {\em now remains bounded for all supercritical (unstable) velocities $U>\Ucrit$}, demonstrating the Lyapunov stability induced by the $-\|w_x\|^2  w_{xx}$ term. From the point of view of trajectories, each remains bounded for all time, with global-in-time bound dependent upon $U$ (and other intrinsic parameters). {\em This confirms the formal conjecture in Remark \ref{alf}.} 

For $U>\Ucrit$ in the nonlinear case ($b_2>0$), we refer to the non-transient behavior of the dynamics as {\em flutter}, and, expect convergence to a limit cycle\footnote{We note that other qualitative properties---including chaotic behavior---of the dynamics are possible when $b_1>>0$ or $p_0(x) \not\equiv 0$, those these cases are not investigated in this treatment.}. For $U<\Ucrit$, we note that the nonlinearity does not dramatically perturb the corresponding exponentially decay of the trajectory observed in the linear case.
\begin{figure}[htp]
\begin{center}
\includegraphics[width=5in]{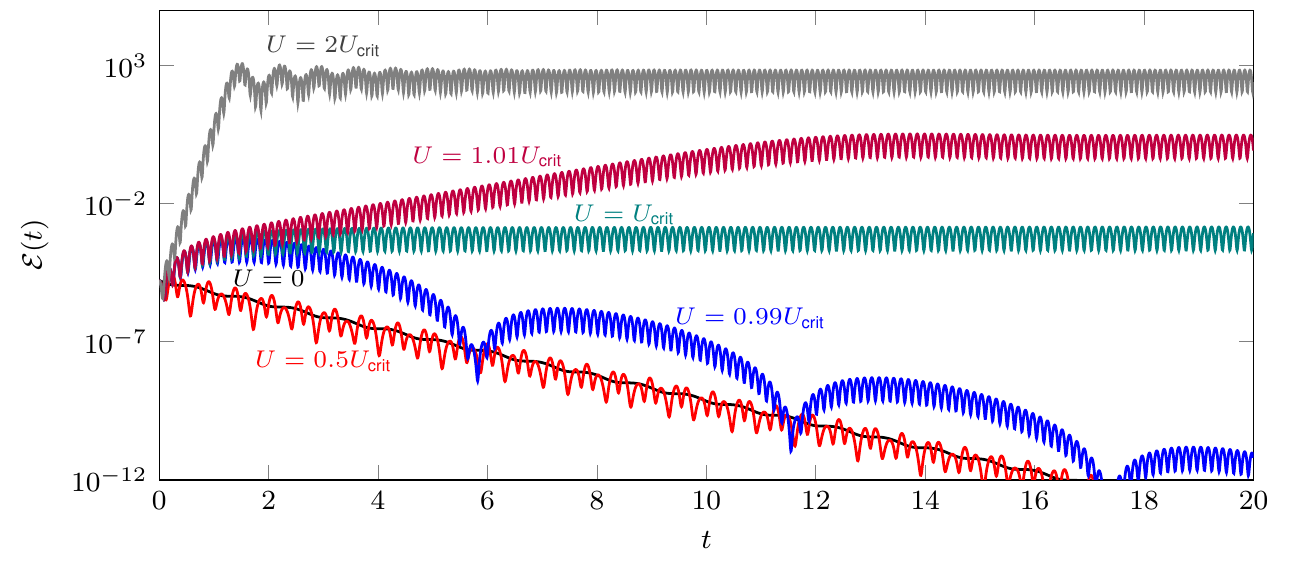}\\
\vspace*{-0.2in}
\caption{Plot of $\mathcal{E}(t)$ for $\alpha=k_1=k_0=0$, and $b_2=1$, varying $U$; $\Ucrit=135.97$.}
\label{fig2}
\end{center}
\end{figure}

\subsubsection{$\alpha>0$: Rotational inertia}
Rotational inertia is incorporated into \eqref{k1} when $\alpha$ is taken to be positive.  As discussed above, the presence of rotational inertia decreases the critical flow speed $\Ucrit (\alpha)$.   For $\alpha=10^{-3}$, the empirically determined critical flow velocity is $\Ucrit(.001)=129.68$.  Analogous to the prior section, computed energies are given in Figure \ref{fig16} for the linear model without nonlinear effects, and $b_2=1$ in Figure \ref{fig17}.  Note that the vertical axis scaling in Figures \ref{fig16} and \ref{fig17} are the same as in Figures \ref{fig1} and \ref{fig2}.  
\begin{figure}[htp]
\begin{center}
\includegraphics[width=5in]{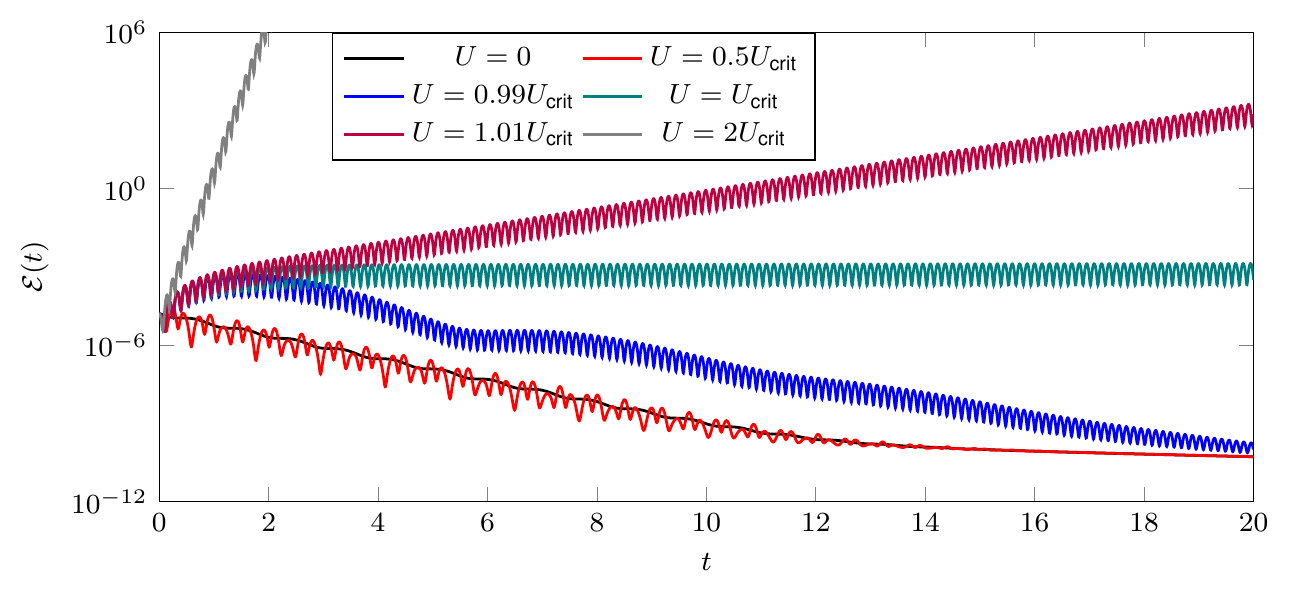}\\
\caption{Plot of $\mathcal{E}(t)$ for $\alpha=10^{-3}$, $k_1=k_0=0$, $b_2=0$, and varying $U$; $U_{\text{crit}}=129.68$.}\label{fig16}
\end{center}
\end{figure}
\begin{figure}[htp]
\begin{center}
\includegraphics[width=5in]{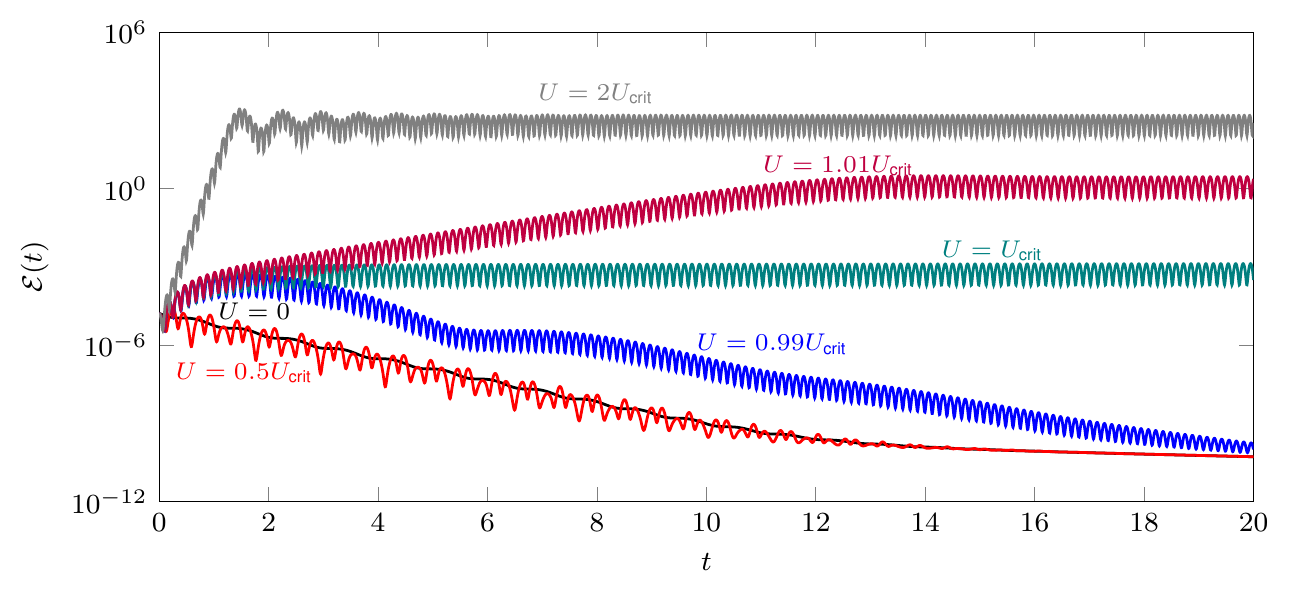}\\
\vspace*{-0.2in}
\caption{Plot of $\mathcal{E}(t)$ for $\alpha=10^{-3}$, $k_1=k_0=0$, and $b_2=1$, varying $U$; $\Ucrit=129.68$.}
\label{fig17}
\end{center}
\end{figure}

Figures \ref{fig13} and \ref{fig14} give trajectories for $\alpha=1$, with $b_2=0$ in Figure \ref{fig13} and $b_2=1$ in Figure \ref{fig14}.  For $\alpha=1$, the empirically determined critical flow velocity is $\Ucrit=22.09$.  When rotational inertia is present at this magnitude, there is less distinction between energy profiles near $\Ucrit$. 
\begin{figure}[htp]
\begin{center}
\includegraphics[width=5in]{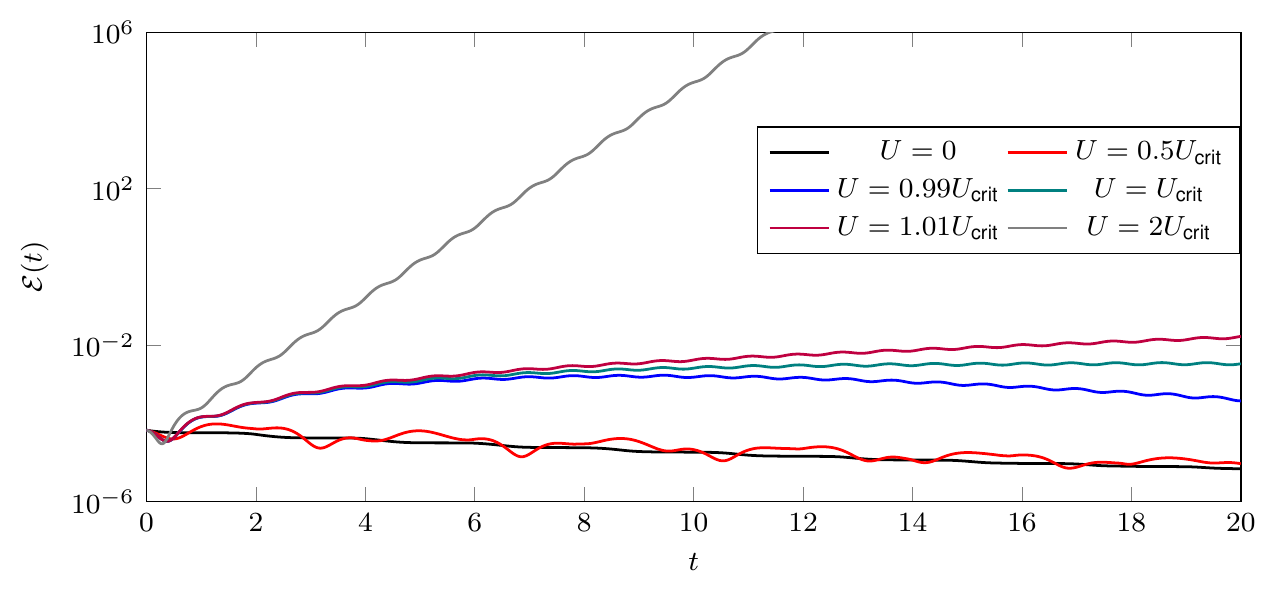}\\
\caption{Plot of $\mathcal{E}(t)$ for $\alpha=1$, $k_1=k_0=0$, $b_2=0$, and varying $U$; $U_{\text{crit}}=22.09$.}\label{fig13}
\end{center}
\end{figure}
\begin{figure}[htp]
\begin{center}
\includegraphics[width=5in]{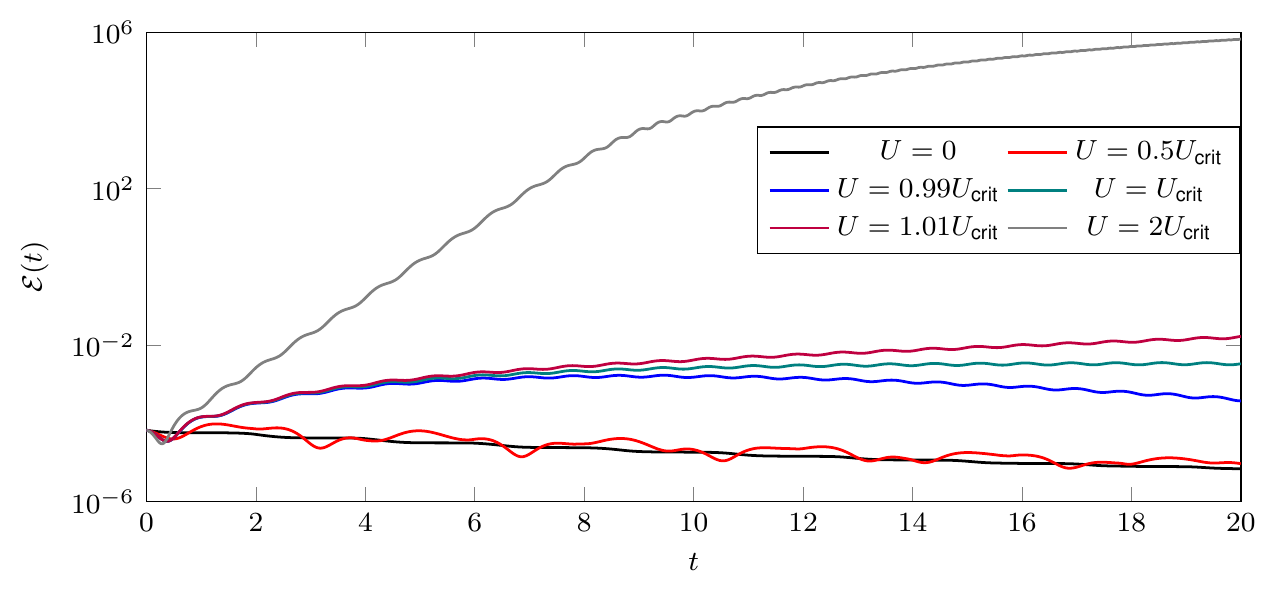}\\
\vspace*{-0.2in}
\caption{Plot of $\mathcal{E}(t)$ for $\alpha=1$, $k_1=k_0=0$, and $b_2=1$, varying $U$; $\Ucrit=22.09$.}
\label{fig14}
\end{center}
\end{figure}
As before, {\em the presence of $b_2>0$ guarantees boundedness of trajectories for all times.}

\subsection{Influence of the nonlinear parameter $b_2$}

Subsequently, the effect of increasing the nonlinear parameter $b_2$ was studied on a supercritical flow velocity ($U >\Ucrit$) for beam with and without rotational inertia.  In Figure \ref{fig11}, energy profiles for several different choices of $b_2$ are given for the parameters $\alpha=k_1=k_0=0$,  and $U=150$ ($\Ucrit=135.97$).  Note that  as $b_2$ increases, the energy plateau decreases (for an initial configuration fixed across all simulations).
\begin{figure}[htp]
\begin{center}
\includegraphics[width=5in]{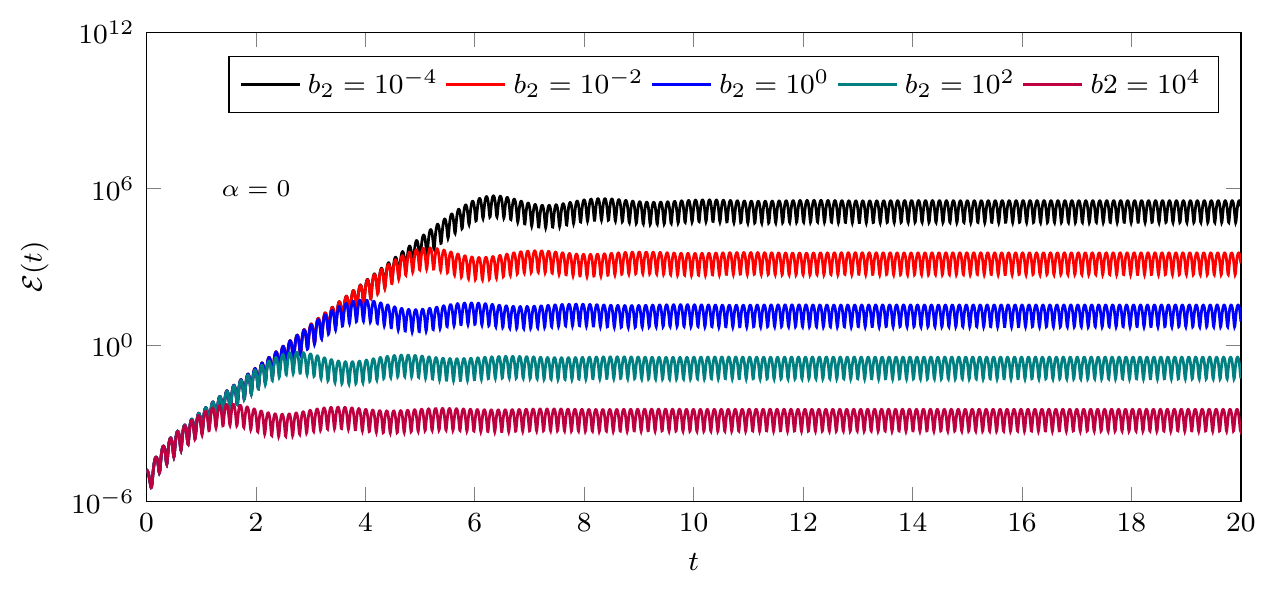}\\
\vspace*{-0.2in}
\caption{Plot of $\mathcal{E}(t)$ for $\alpha=k_1=k_0=0$, and $U=150$, varying $b_2$.}
\label{fig11}
\end{center}
\end{figure}

In Figure \ref{fig12}, energy profiles for several different choices of $b_2$ are computed with $\alpha=10^{-3}$ and $U=150$ (the critical flow velocity for this $\alpha$ is $129.68$).  The vertical axis is scaled as in Figure \ref{fig11}.  Note that, in contrast to Figure \ref{fig11}, energies for beams with $\alpha>0$ initially grow at a faster rate (as expected) than when rotational inertia is absent; since the energy plateaus are comparable for both values of $\alpha$, we observe the effect of $\alpha$ ``speeding up" the bounding effect induced by the presence of nonlinearity. 

\begin{figure}[htp]
\begin{center}
\includegraphics[width=5in]{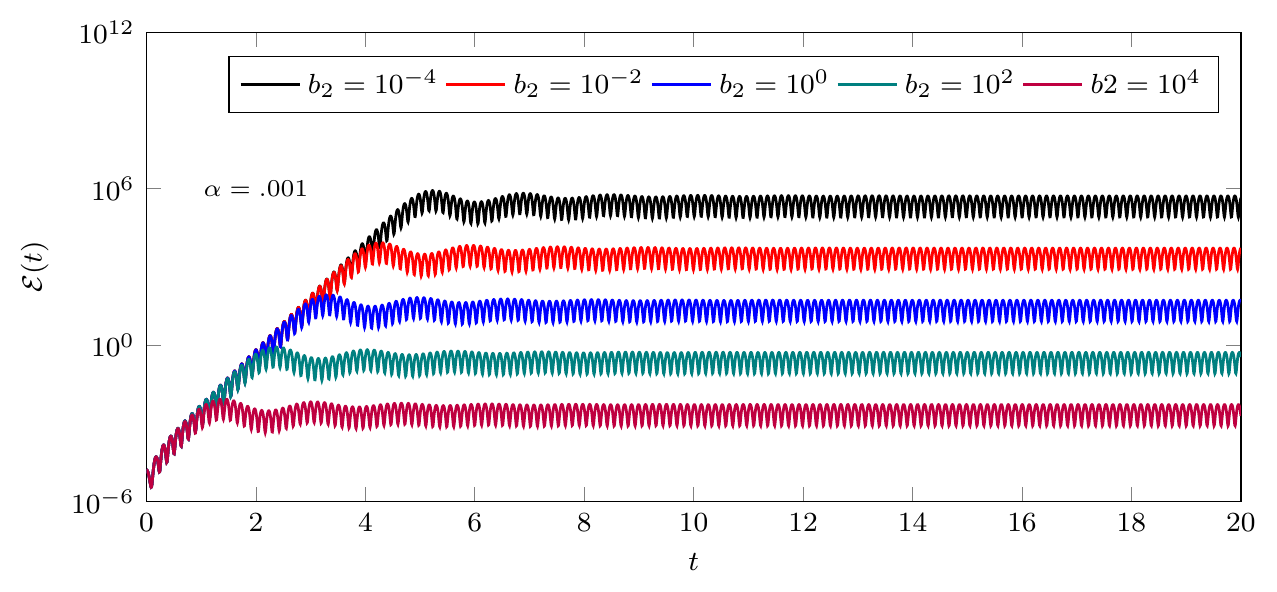}\\
\vspace*{-0.2in}
\caption{Plot of $\mathcal{E}(t)$ for $\alpha=10^{-3}$, $k_1=k_0=0$, and $U=150$, varying $b_2$.}
\label{fig12}
\end{center}
\end{figure}

Finally, \ref{fig18} presents energy profiles for the case when $\alpha=1$ and $U=50$, which is greater than the critical velocity of $22.09$ for this $\alpha$.  Note that the energy profiles are much smoother, but the energies are greater, even with a substantially smaller $U$ value.

\begin{figure}[htp]
\begin{center}
\includegraphics[width=5in]{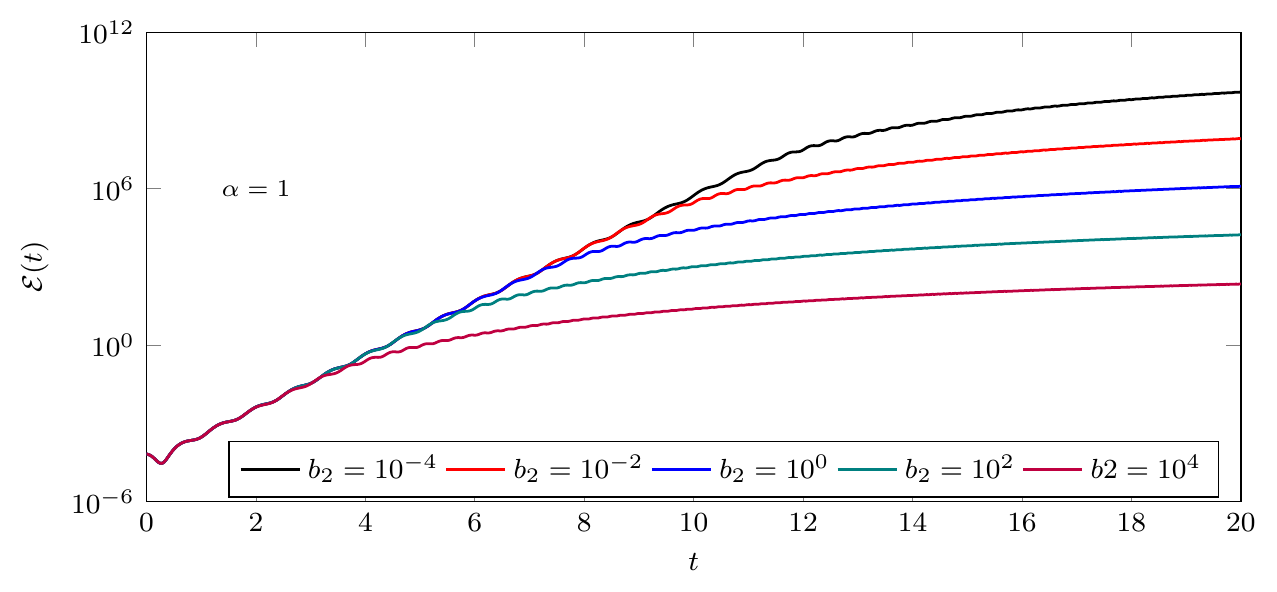}\\
\vspace*{-0.2in}
\caption{Plot of $\mathcal{E}(t)$ for $\alpha=1$, $k_1=k_0=0$, and $U=50$, varying $b_2$.}
\label{fig18}
\end{center}
\end{figure}

\subsection{Stability when increasing $\alpha=k_1$}
The collective influence of the terms in \eqref{k1} that pertain to rotational inertia is  studied by fixing all other parameters and allowing $\alpha=k_1$ to vary.  In Figure \ref{fig7}, energy plots for $\alpha=k_1$ ranging from $10^{-4}$ to $10^{4}$ with $b_1=b_2=0$, $k_0=0$, and $U=150$ (supercritical) are given.  For values of $\alpha=k_1$ from $10^{-4}$ to $10^{-1}$, the slope of the energy profile increases, indicating that the rotational inertia terms decrease the linear stability of the beam.  However, the slope of the energy profiles decreases around $10^{-1}$, and eventually regains stability, although there is a slight increase in the maximum computed energy from $10^{2}$ to $10^{4}$. \emph{These curves represent the trade-off between the destabilizing effect of increasing $\alpha$ on the beam's eigenvalues, and the stabilizing effect of strong damping, due to increasing $k_1$ when $\alpha>0$}.
\begin{figure}[htp]
\begin{center}
\includegraphics[width=5in]{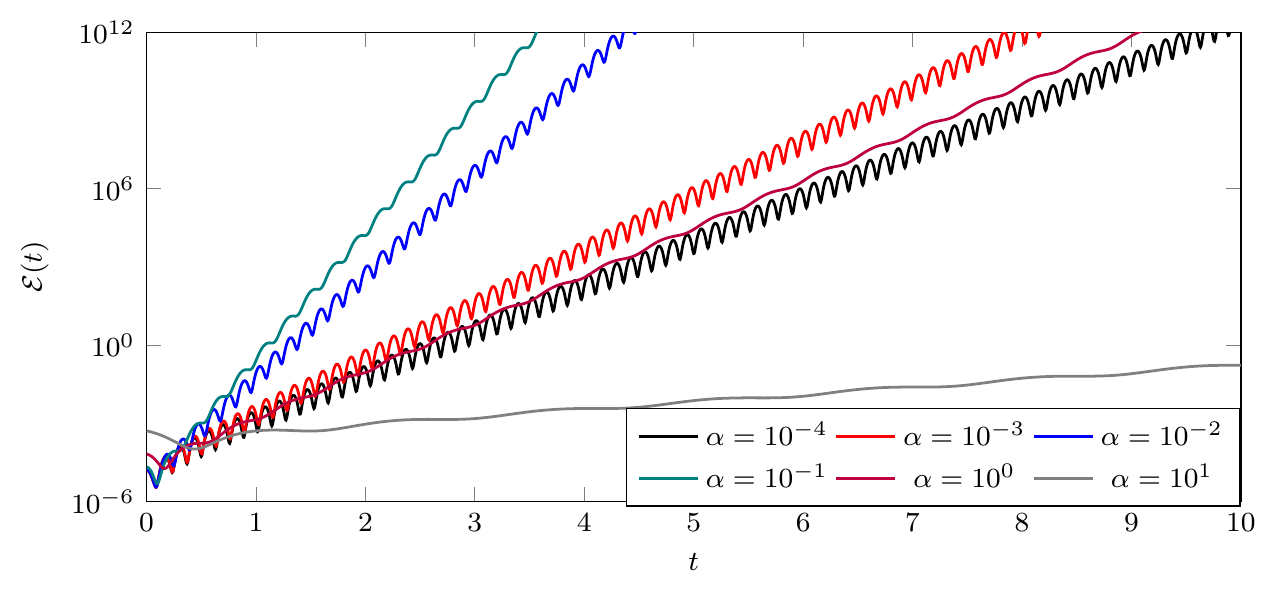}\\
\vspace*{-0.2in}
\caption{Plot of $\mathcal{E}(t)$ for varying $\alpha=k_1$, $b_2=0$,  $U=150$, $k_0=0$.}
\label{fig7}
\end{center}
\end{figure}

To better understand the overall behavior seen in Figure \ref{fig7}, we examine how $\emax$ behaves when $\alpha=k_1$ is varied.  In Figure \ref{fig8}, the energy max at $T=20$ is computed along a fine grid of $\alpha=k_1$ values ranging from $10^{-4}$ to $10^{4}$ (with all other parameters the same as in Figure \ref{fig7}).  It is observed that $\emax(20)$ peaks at around $\alpha=k_1=10^{-1}$ and subsequently decreases until approximately $10^{1.5}$.  At $\alpha=k_1\approx 10^{1.5}$ the curve exhibits a mild cusp and then increases slowly to $\alpha=k_1=10^{4}$.  
\begin{figure}[htp]
\begin{center}
\includegraphics[width=5in]{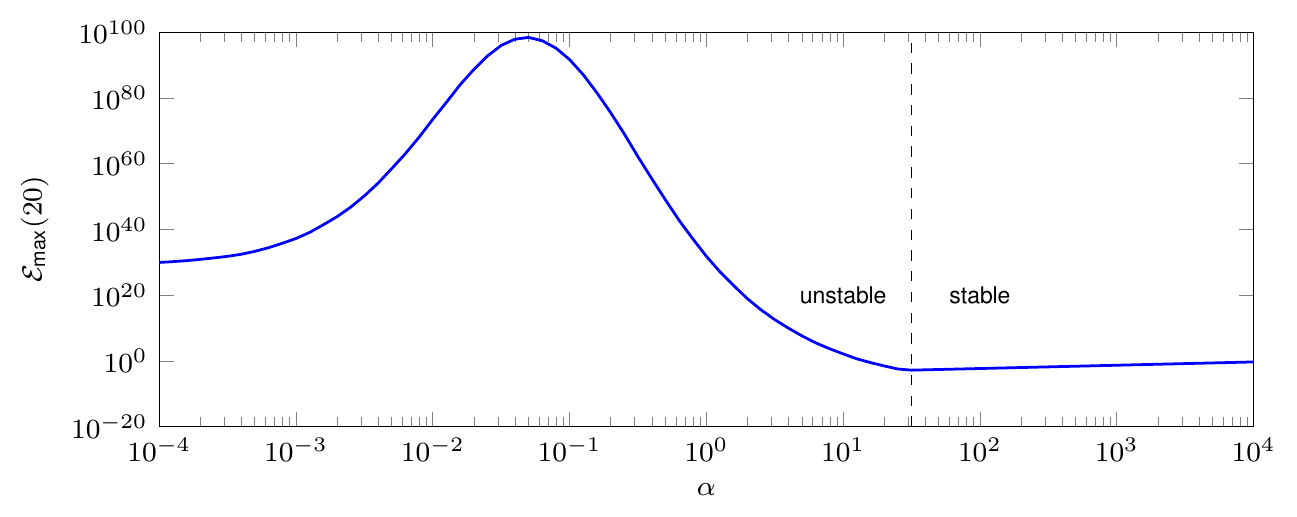}\\
\vspace*{-0.2in}
\caption{Plot of $\emax(0,20)$ for varying $\alpha=k_1$, $b_1=b_2=0$,  $U=150$, $k_0=0$.}
\label{fig8}
\end{center}
\end{figure}
A plot of the beam tip displacement $w(L,t)$ is given for $\alpha=k_1=1.4, 1.5, 1.6$ in Figure \ref{fig9}.  All three exhibit comparable oscillations for (approximately) $0<t<3$, but the collective effect of increasing $\alpha=k_1$ is to suppress amplitude growth and extend the period of oscillation. It is clear that, in the smallest case of $\alpha=k_1=1.4$, the displacement amplitude is increasing.
\begin{figure}[htp]
\begin{center}
\includegraphics[width=5in]{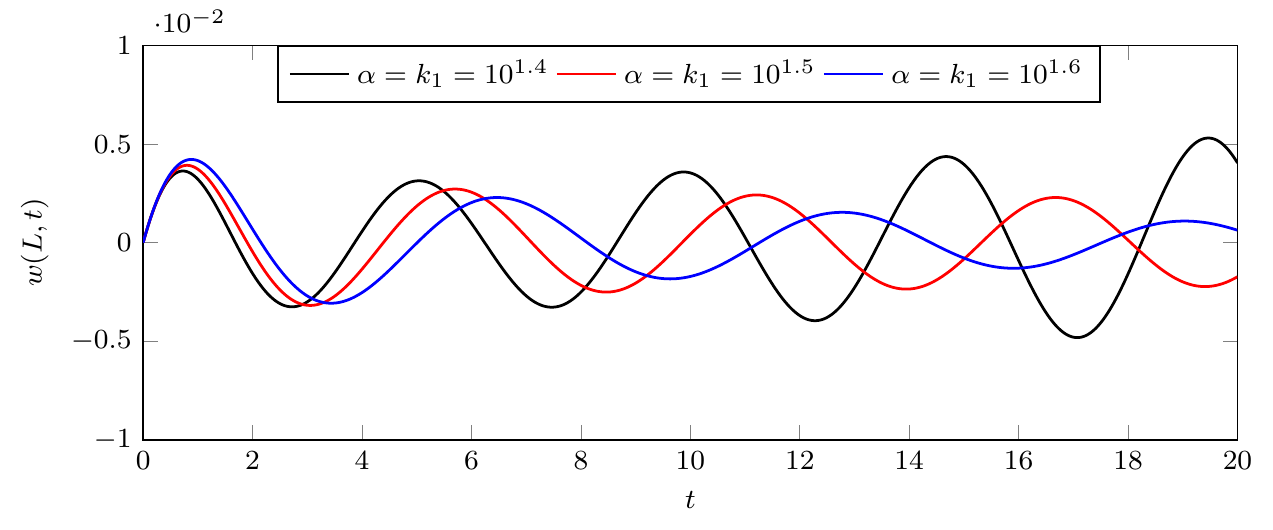}\\
\vspace*{-0.2in}
\caption{Plot of $w(L,t)$ for varying $\alpha=k_1$, $b_1=b_2=0$,  $U=150$, $k_0=0$.}
\label{fig9}
\end{center}
\end{figure}

\subsection{Qualitative effects of increasing rotational inertia}

To understand the influence of the rotational inertia term $-\alpha \partial_x^2w_{tt}$ by itself on the dynamics of an unstable beam, energies were computed for a wide range of $\alpha$.  Here both damping parameters $k_0$ and $k_1$ are set to zero, $b_2=0$, and $U=150$ (supercritical).  The energy profiles are given in Figure \ref{fig10}.  For values of $\alpha$ up to around $10^{-1}$, the profiles are very similar to those observed above in Figure \ref{fig7} where $k_1=\alpha$.  Energies for $\alpha > 10^{-1}$ are (in general) larger than those in Figure \ref{fig7}, due to the exclusion of the strong damping term $-k_1\partial_x^2w_{t}$ here.
\begin{figure}[htp]
\begin{center}
\includegraphics[width=5in]{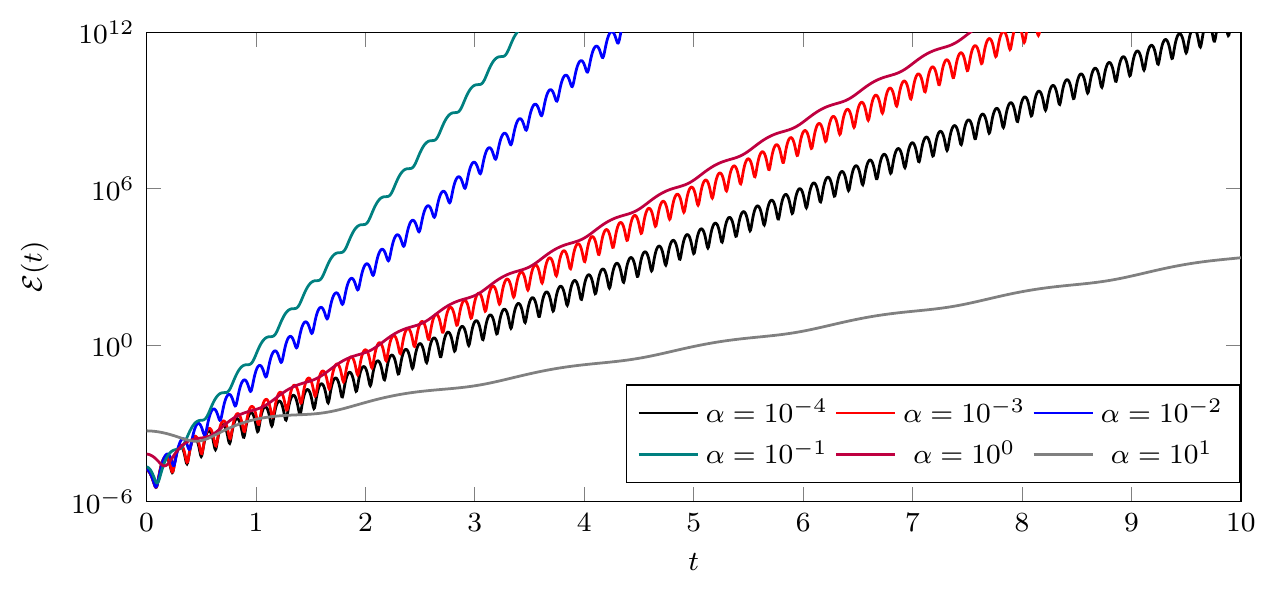}\\
\vspace*{-0.2in}
\caption{Plot of $\mathcal{E}(t)$ for  varying $\alpha$, $b_2=0$,  $U=150$, $k_0=k_1=0$.}
\label{fig10}
\end{center}
\end{figure}
In Figure \ref{fig15}, the same energies are computed as in Figure \ref{fig7} ($k_0=0, U=150)$  but with $b_2=1$.  
\begin{figure}[htp]
\begin{center}
\includegraphics[width=5in]{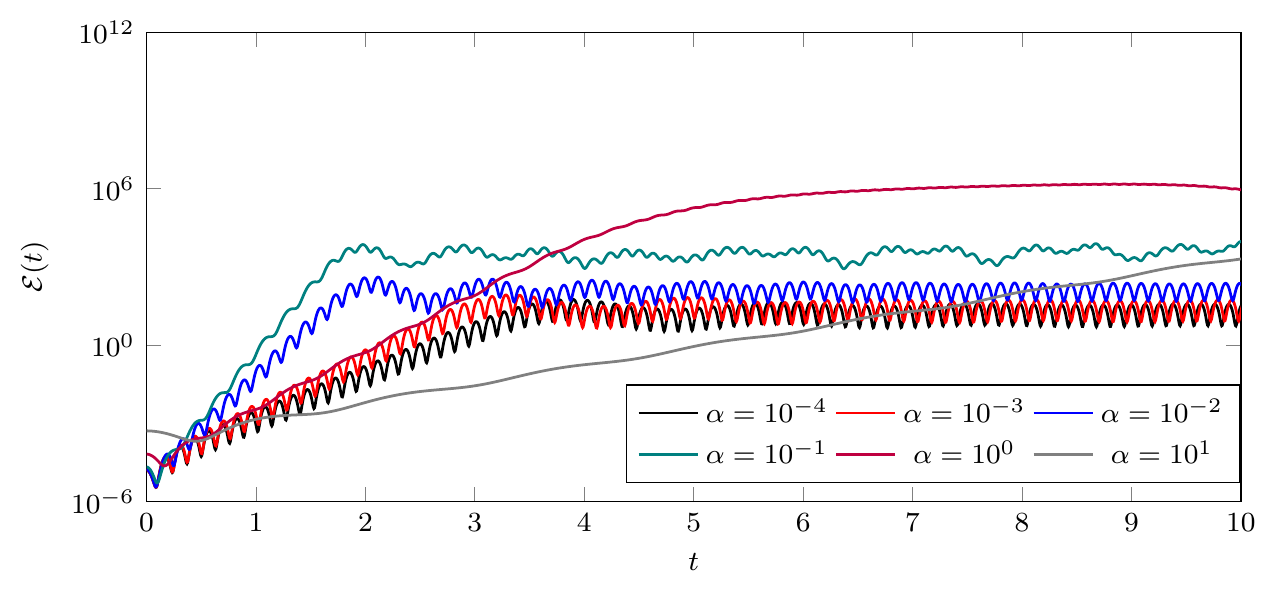}\\
\vspace*{-0.2in}
\caption{Plot of $\mathcal{E}(t)$ for varying $\alpha=k_1$, $b_1=0$, $b_2=1$,  $U=150$, $k_0=0$.}
\label{fig15}
\end{center}
\end{figure}
In these simulations, it appears that the size of $\alpha$ impacts the periodicity of the nonlinear energy in the nonlinear regime (LCO).

In Figure \ref{fig3}, the plot shows the effect of including rotational inertia in the model \eqref{k1} on trajectories of a supercritical flow velocity of $U=150$.    For $b_2=1$ and $k_1=0$, the rotational inertia delays the stabilizing effect of the nonlinearity, but note that the amplitude of the energy oscillations is somewhat smaller than without rotational inertia.  As expected via \eqref{def:energy}, the inclusion of rotational inertia increases the overall profile of $\mathcal{E}(t)$.
\begin{figure}[htp]
\begin{center}
\includegraphics{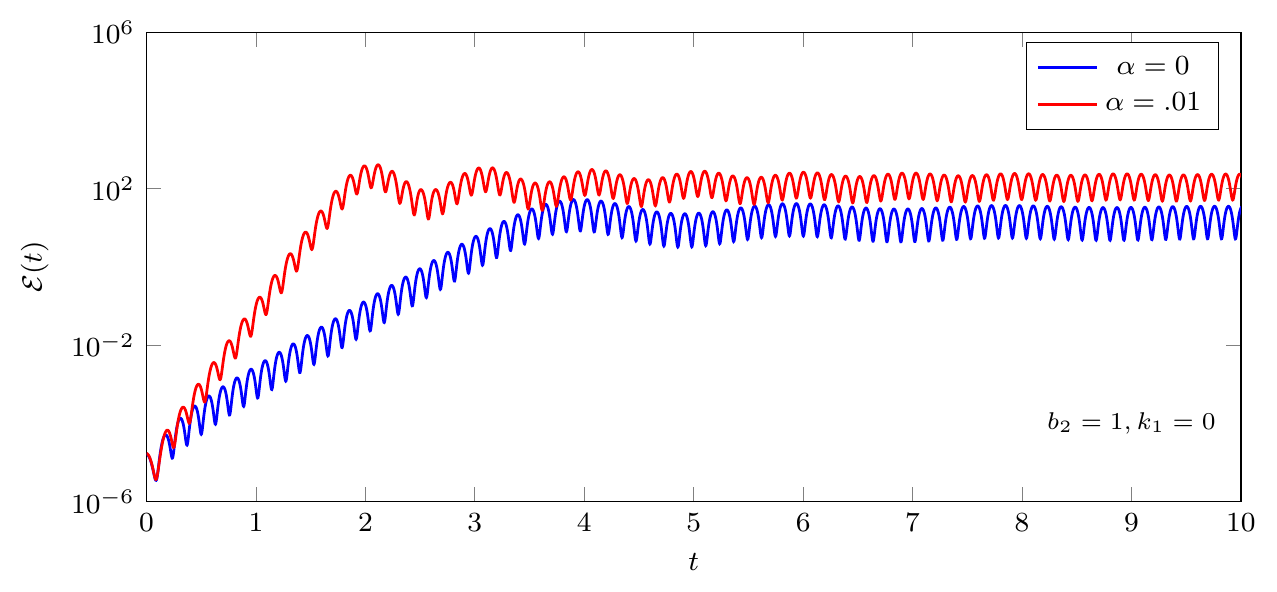}
\\
\vspace*{-0.2in}
\caption{Plot of $\mathcal{E}(t)$ for $k_0=k_1=0$, $U=150$, $b_2=1$, varying $\alpha$.}
\label{fig3}
\end{center}
\end{figure}
The effect of including rotational inertia can also be observed when examining $w_t(L)$ for a particular trajectory.  In Figure \ref{fig4}, plots of the velocity at the beam endpoint ($x=L$) are given for the same cases in Figure \ref{fig3}. 
\begin{figure}[htp]
\begin{center}
\includegraphics{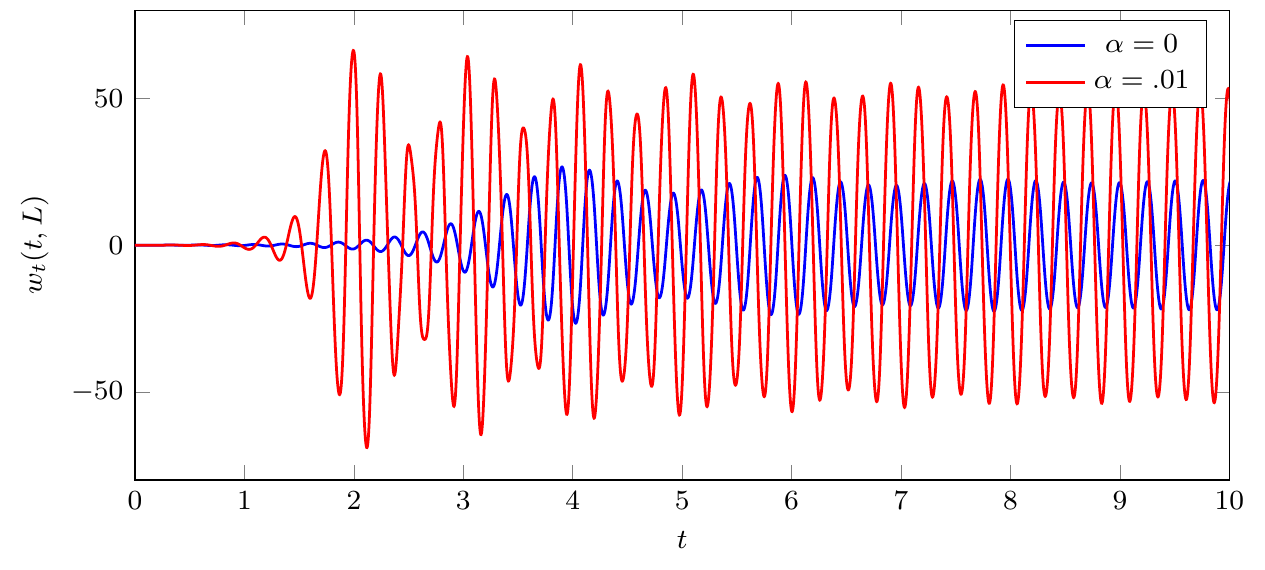}
\\
\vspace*{-0.2in}
\caption{Plot of $w_t(L,t)$ for $k_0=k_1=0$, $U=150$, $b_2=1$, varying $\alpha$.}
\label{fig4}
\end{center}
\end{figure}
\subsection{Comparative effects of strong and viscous damping}

The two damping coefficients $k_0$ (viscous) and $k_1$ (strong) in \eqref{k1} influence the stability of trajectories in different ways (see Remark \ref{sqrt}). The quantity $k=k_0+\beta$ can be considered as a total viscous damping parameter for \eqref{k1};  $\beta=1$ represents the scaling of the frictional damping due to the piston-theoretic term (that is, due to flow effects), and $k_0$ represents the material/imposed damping.  Changes in the total damping are effected via changes in $k_0$ and/or $k_1$.

In Figure \ref{fig5}, computed energies are plotted for several different choices of  strong damping coefficient $k_1$ with $k_0=0$ at (the unstable) $U=150$.  No nonlinear effects are considered ($b_2=0$), and $\alpha$ is fixed at $10^{-3}$.  Note that trajectories become {\em stable} as $k_1$ increases, with a ``critical'' $k_1$ somewhere near $2$ at these parameter values. 
\begin{figure}[htp]
\begin{center}
\includegraphics[width=5in]{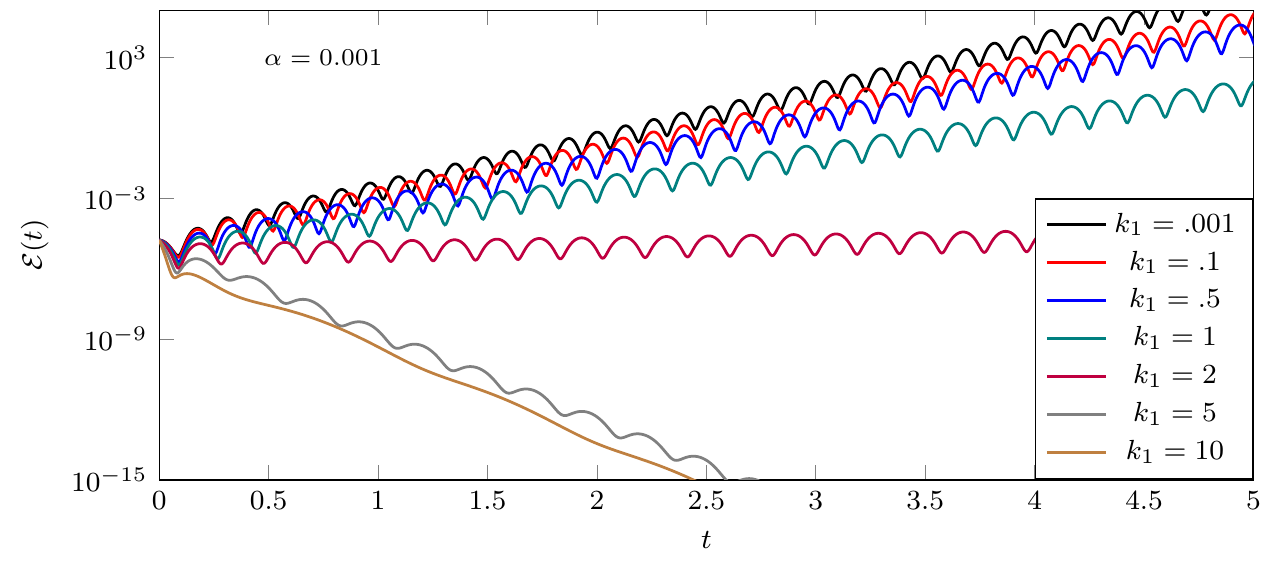}\\
\vspace*{-0.2in}
\caption{Plot of $\mathcal{E}(t)$ for $\alpha=10^{-3}$, $b_2=0$,  $U=150$, $k_0=0$, varying $k_1$.}
\label{fig5}
\end{center}
\end{figure}

In Figure \ref{fig6}, computed energies are plotted for several different choices of  viscous damping coefficient $k_0$, with $k_1=0$, also at $U=150$.  Again, $b_2=0$ is considered and the rotational inertia parameter $\alpha=10^{-3}$.  Note that a ``critical'' value of $k_0$ occurs near $5$, suggesting that {\em even in the presence of rotational inertia, it is possible to stabilize energies via viscous damping alone.}
\begin{figure}[htp]
\begin{center}
\includegraphics[width=5in]{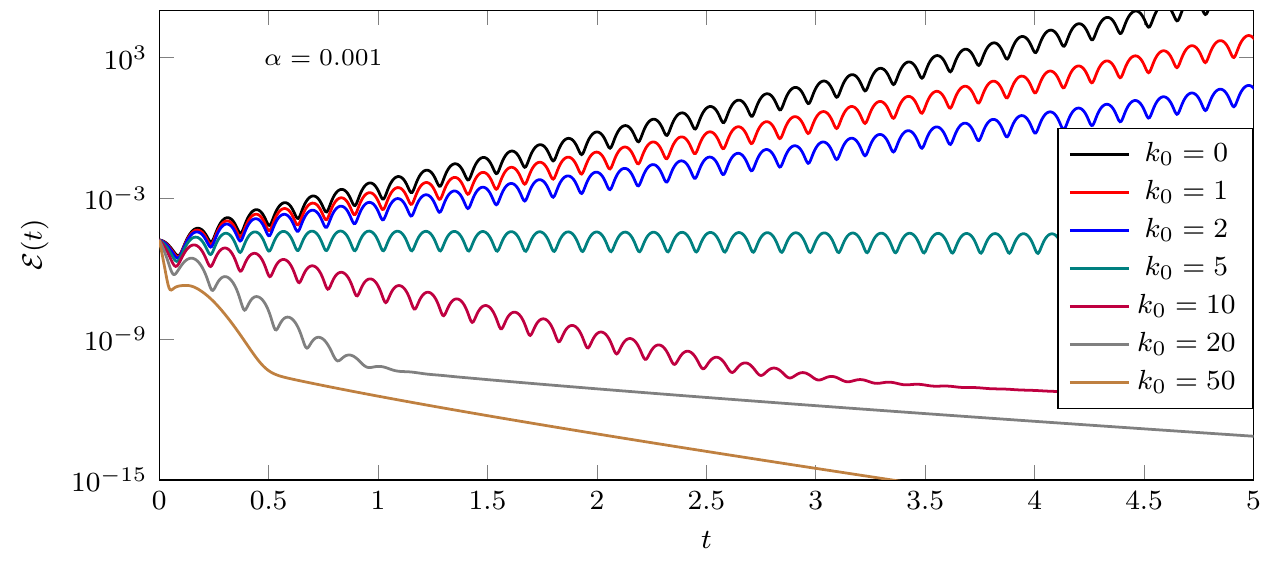}\\
\vspace*{-0.2in}
\caption{Plot of $\mathcal{E}(t)$ for $\alpha=10^{-3}$, $b_2=0$,  $U=150$, $k_1=0$, varying $k_0$. }
\label{fig6}
\end{center}
\end{figure}

The sensitivity of the dynamics to the damping parameters can be seen by noting the relative rate of initial decay or growth of energy as $k_0$ or $k_1$ increases (with other parameters fixed), and the relative sizes of the damping parameters and their impact on stability. On a trajectory by trajectory basis for a fluttering beam, one can see the effect of the damping parameters $k_0$ and $k_1$ on the LCO, as well as on the decay of the transient dynamics---Figures \ref{fig19}--\ref{fig20}.
\begin{figure}[h!]
\centering
\includegraphics{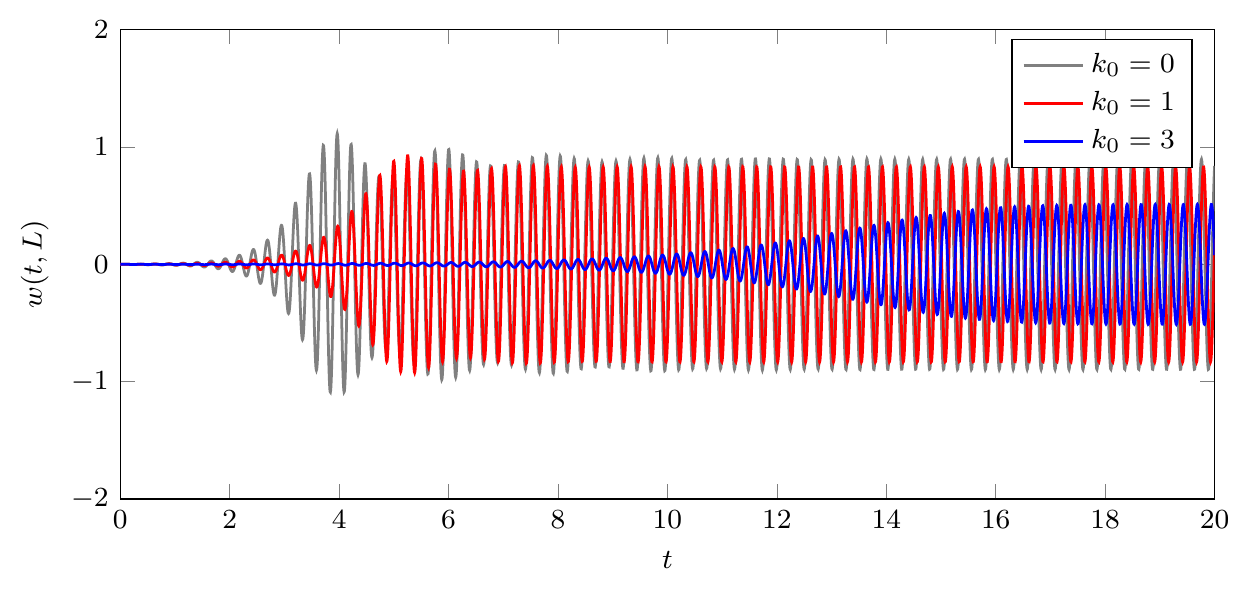} 
\caption{Plot of $w$ at beam endpoint for $U=150$, $b_2=1$, $\alpha=0$, $k_1=0$, varying $k_0$.}
\label{fig19}
\end{figure}
\begin{figure}[h!]
\centering
\includegraphics{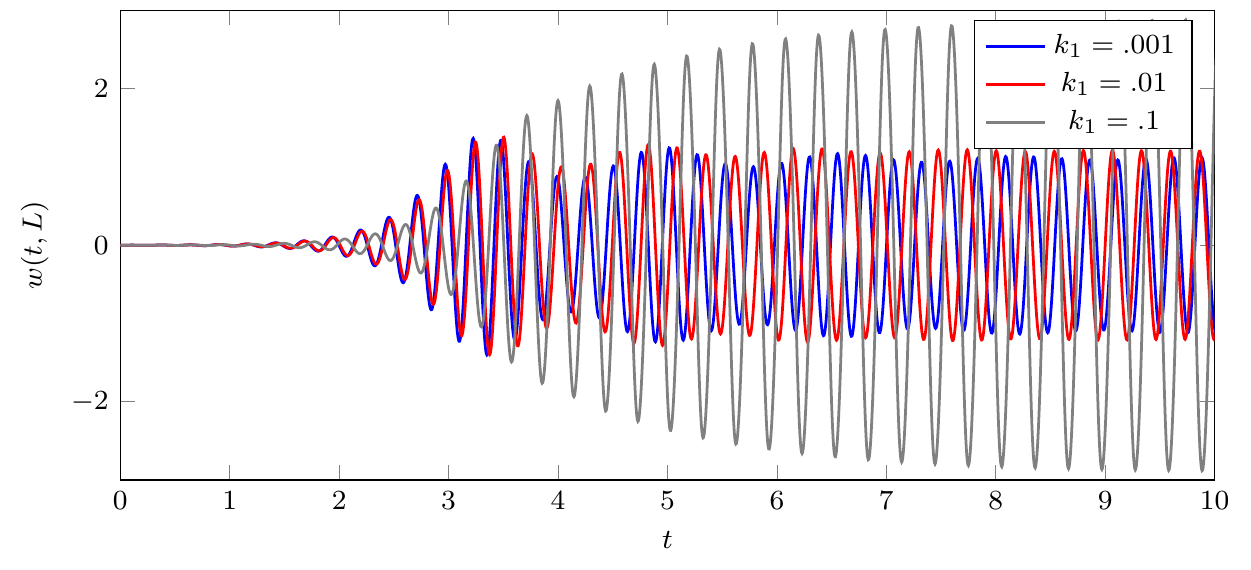}
\caption{Plot of $w$ at beam endpoint for $U=150$, $b_2=1$, $\alpha=.001$, $k_0=0$, varying $k_1$.}
\label{fig20}
\end{figure}
\begin{remark} We conclude with a remark about the somewhat peculiar case of $\alpha=0$ and $k_1>0$. In this case we do not include rotational inertia effects, but {\em do} include strong (or ``square root" type) damping. Preliminary simulations for an unstable configuration (e.g., $U=150$) for the linear beam indicate that increasing the strong damping coefficient $k_1$ in this case has a much different effect on {\em stability} properties of the dynamics than increasing viscous damping coefficient $k_0$. In particular, two separate effects are observed: (i) with all other parameters fixed, relatively small values of $k_1$ will result in stabilization of the dynamics, in contrast to utilizing $k_0$ for stability, which on its own must be large outright to yield a stabilizing effect;  (ii) however, as $k_1$ increases, the so called {\em margin of stability} approaches zero, which is not observed for larger values of $k_0$. In general, the margin of stability for large $k_0$ is much greater than the corresponding margin for large $k_1$.
\end{remark}

\subsection{Synopsis of main numerical observations}
In this section we briefly provide a synopsis of the most pertinent conclusions/observations from the simulations presented above
\begin{itemize}
\item The flutter point $U_{\text{crit}}$ decreases as $\alpha\ge 0$ increases. This is in line with the general notion that $\alpha>0$ is a destabilizing presence for the in vacuo beam. 
\item The emergence of instability via $U>\Ucrit$ is only marginally affected by the presence of the Berger nonlinearity. In this way, we may claim that the onset of flutter is a purely linear phenomenon, regarding the manner in which $Uw_x$ perturbs the in vacuo eigenmodes of the system; however, to view the flutter dynamics, one must include nonlinear restoring force ($b_2>0$).
\item For fixed intrinsic parameters, we typically observe convergence to the ``same" LCO, across various initial configurations.
\item The nonlinearity considered here (with $b_2>0$ any size) is ``strong" enough to bound trajectories for all values of $\alpha \ge 0$. The overall bound, as the transient dynamics decay, seems uniform in the initial data, and the magnitude of the non-transient behavior (typically an LCO) depends on the intrinsic parameters, such as $b_2, U$, etc.
\item Stable and unstable nonlinear dynamics seem to converge as $\alpha \searrow 0$. There are no apparent issues with the $\alpha=0$ dynamics, though as discussed in the theoretical part of the treatment, we cannot prove that a proper dynamical system exists in this case, nor can we show the existence of a compact global attractor encapsulating the flutter behavior. 
\item Viscous damping of the form $k_0w_t$ (for $k_0$ sufficiently large) is ``strong" enough to prevent/stabilize flutter. For a fixed $U>\Ucrit$, increasing $k_0$ provides two regimes: initially, it increases the rate of convergence to a single stable LCO, but, for $k_0$ sufficiently large, the flutter dynamics are eventually damped out. 

Strong damping, via the $k_1$ coefficient, stabilizes in the same manner as described above for $k_0$ {\em so long as $\alpha>0$}. 
\item When $\alpha=k_1$ are scaled up together, there is a competition: $\alpha>0$ is destabilizing for the dynamics, but eventually $k_1$ takes over and re-stabilizes the dynamics. 
\end{itemize}

\section{Acknowledgments}
D. Toundykov's research was partially supported by the National Science Foundation with grant NSF-DMS-1616425.
 J.T. Webster's research was partially supported by the National Science Foundation with grant NSF-DMS-1504697.

\section{Appendix: Long-time behavior of dynamical systems}
 Let $(H,S_t)$ be a dynamical system on a complete metric space $H$.    $(H,S_t)$ is said to be (ultimately) dissipative iff it
possesses a bounded absorbing set $\mathcal B$. This is to say
that for any bounded set $D$, there is a time $t_D$ so that
$S_{t_D}(D)\subset \mathcal B$. 

We say that a dynamical system is
\textit{asymptotically compact} if there exists a compact set $K$
which is uniformly attracting: for any bounded set $ D\subset
H$ we have that $$\displaystyle~
\lim_{t\to+\infty}d_{{H}}\{S_t D|K\}=0$$ 
in the sense of
the Hausdorff semidistance. 
$(H,S_t)$ is said to be
\textit{asymptotically smooth} if for any bounded, forward
invariant $(t>0) $ set $D$ there exists a compact set $K \subset
\overline{D}$ which is uniformly attracting (as above).
\begin{theorem}[Proposition 7.1.4 \cite{springer}]\label{compact}
For a dissipative dynamical system $(H,S_t)$, asymptotic smoothness and asymptotic compactness are equivalent. 
\end{theorem}
A \textit{global attractor} $A\subset H$ is a closed, bounded set in $H$ which is (fully) invariant (i.e. $S_tA={A}$ for all $
t>0 $) and uniformly attracting (as defined above). 

\begin{remark} Since we are considering a dynamics that are inherently non-gradient, we do not (here) discuss the set of stationary points, strict Lyapunov functions, or certain characterizations available for dynamical systems that are gradient (see \cite[Chaper 7]{springer}). \end{remark}

\begin{theorem}[Theorem 7.2.3 \cite{springer}]\label{dissmooth}
Let $(H,S_t)$ be an dissipative dynamical system in a complete
metric space  $H$. Then $(H,S_t)$ possesses a compact
global attractor ${A} $ if and only if $(H,S_t)$ is
asymptotically smooth.
\end{theorem}

For non-gradient systems, this theorem is often the mechanism employed to obtain the existence of a compact global attractor. If one can show that a dissipative dynamical system $(H,S_t)$ is asymptotically smooth, one obtains the existence of a compact global attractor. In many cases, showing asymptotic smoothness can be done conveniently using the criterion due to \cite{kh} and presented in a streamlined way in \cite{springer}.

\begin{theorem}[Theorem 7.1.11 \cite{springer}]
\label{psi} Let $(H,S_t)$ be a dynamical system,
$H$ a Banach space with norm $\|\cdot\|$. Assume that for any
bounded positively invariant set $D \subset H$ and for
all $\epsilon>0$ there exists a $T\equiv T_{\epsilon,D}$ such that
\begin{equation*}
\|S_Ty_1 - S_Ty_2\|_{\mathcal{H}} \le
\epsilon+\Psi_{\epsilon,B,T}(y_1,y_2),~~y_i \in D
\end{equation*}
with $\Psi$ a functional defined on $D \times D$ depending on
$\epsilon, T,$ and $D$ such that
\begin{equation*}
\liminf_m \liminf_n \Psi_{\epsilon,T,D}(x_m,x_n) = 0
\end{equation*}
for every sequence $\{x_n\}\subset D$. Then $(H,S_t)$ is
 asymptotically smooth.
\end{theorem}

\footnotesize
\bibliographystyle{abbrv} 
\bibliography{flutter.bib}
\end{document}

\end{document}